\documentclass[a4paper,12pt]{article}
\usepackage{amsmath}
\usepackage{amsfonts}
\usepackage{amssymb}
\usepackage{amsthm}
\usepackage{mathrsfs}

\newtheorem{theorem}{\bf Theorem}[section]
\newtheorem{lemma}[theorem]{\bf Lemma}
\newtheorem{corollary}[theorem]{\bf Corollary}
\newtheorem{definition}[theorem]{\bf Definition}
\newtheorem{remark}[theorem]{\bf Remark}
\newtheorem{proposition}[theorem]{ Proposition}
\newtheorem{example}[theorem]{\bf Example}

\newcommand{ \mis}{\textrm{meas}}  
\newcommand{\pr}[1]{
\begin{proof}
#1
\end{proof}
}
\newcommand{\es}[2]{
\begin{equation}\label{#1}
\begin{split}
#2
\end{split}
\end{equation}
}
\newcommand{\en}[1]{
\begin{equation*}
\begin{split}
#1
\end{split}
\end{equation*}
}
\usepackage[english]{babel}

\def \k {{\kappa}}

\def \b {{\beta}}
\def \d {{\delta}}

\def \L {\mathscr{L}}
\def \LL {{\widetilde {\mathscr{L}}}}
\def \K {\mathscr{K}}

\def \OO {{\mathbb{O}}}
\def \R {{\mathbb {R}}}
\def \N {{\mathbb {N}}}

\def \x {{\xi}}

\def \t {{\tau}}

\def \z {{\zeta}}

\def \div {{\text{\rm div}}}

\def \tilde {\widetilde}

\def \rn {{\mathbb {R}}^{N}}
\def \rnn {{\mathbb {R}}^{N+1}}

\def \H {\mathcal{Q}}
\begin{document}

\title{\bf Schauder type estimates for degenerate Kolmogorov equations with Dini continuous coefficients}

\author{
{\sc{Sergio Polidoro\thanks{Dipartimento di Scienze Fisiche, Informatiche e Matematiche, Universit\`{a} di Modena e Reggio Emilia, via Campi 213/b, 41115 Modena (Italy). E-mail: sergio.polidoro@unimore.it}}} 
\and 
{\sc{Annalaura Rebucci\thanks{Dipartimento di Scienze Fisiche, Matematiche e Informatiche, Universit\`{a} di Parma, Parco Area delle Scienze 7/A, 43124 Parma (Italy). Email: annalaura.rebucci@unipr.it}}}
\and
{\sc {Bianca Stroffolini \thanks{ Dipartimento di Ingegneria Elettrica e delle Tecnologie dell'Informazione, Universit\`{a} di Napoli Federico II, Via Claudio 25, 80125 Napoli, Italy.  E-mail: bstroffo@unina.it}}}}
\date{}
\maketitle
\begin{abstract}
\noindent We study the regularity properties of the second order linear operator in $\R^{N+1}$:
\begin{equation*}
\L u := \sum_{j,k= 1}^{m} a_{jk}\partial_{x_j x_k}^2 u + \sum_{j,k= 1}^{N} b_{jk}x_k \partial_{x_j} u - \partial_t u,
\end{equation*}
where $A = \left( a_{jk} \right)_{j,k= 1, \dots, m}, B= \left( b_{jk} \right)_{j,k= 1, \dots, N}$ are real valued matrices with constant coefficients, with $A$ symmetric and strictly positive. We prove that, if the operator $\L$ satisfies H\"ormander's hypoellipticity condition, and $f$ is a Dini continuous function, then the second order derivatives of the solution $u$ to the equation $\L u = f$ are Dini continuous functions as well. We also consider the case of Dini continuous coefficients $a_{jk}$'s. A key step in our proof is a Taylor formula for classical solutions to $\L u = f$ that we establish under minimal regularity assumptions on $u$.

\medskip

\noindent

\medskip

\end{abstract}

\setcounter{equation}{0}\setcounter{theorem}{0}
\section{Introduction}
We consider second order linear differential operators of the form 
\begin{equation} \label{e-Kolm-costc}
   \L :=\sum_{i,j=1}^m a_{ij}\partial^{2}_{x_i x_j}+\sum_{i,j=1}^N b_{ij}x_j\partial_{x_i}-\partial_t,
\end{equation}
where $(x,t) \in \R^{N+1}$, and $1 \le m \le N$. 
The matrices $A:=(a_{ij})_{i,j=1,\ldots,m}$ and $B:=(b_{ij})_{i,j=1,\ldots,N}$ have real constant entries. 
The first order part of the operator $\L$ will be denoted by $Y$
\begin{equation} \label{eq-Y}
 Y := \sum_{i,j=1}^N b_{ij}x_j\partial_{x_i}-\partial_t = \langle B x, D \rangle-\partial_t,
\end{equation}
and $Y u$ will be understood as the \emph{Lie derivative}
\begin{equation} \label{eq-Yu}
 Yu (x,t) := \lim_{s \to 0} \frac{u(\text{exp} (s B )x, t-s) - u(x,t)}{s}.
\end{equation}

Note that $Y u$ is the derivative of $u$ along the characteristic trajectory of $Y$, if we identify the directional derivative $Y$ with the vector valued function $Y(x,t) = (Bx,-1)$. The standing assumption of this article is:

\medskip
\noindent
{\bf[H.1]} \ The matrix $A$ is symmetric and strictly positive, the matrix $B$ has the form
   \begin{equation}
 			   \label{B}
 			   B =    \begin{pmatrix}
         B_{0,0} &    B_{0,1}  & \ldots &     B_{0, \k - 1}  &   B_{0, \k }  \\  
        B_1   &    B_{1,1}  & \ldots &     B_{ 1,\k - 1}  &   B_{1 , \k}  \\
        \OO    &    B_2  & \ldots &   B_{2, \k - 1}   &  B_{2 , \k}   \\
        \vdots & \vdots & \ddots & \vdots & \vdots \\
        \OO    &  \OO    &    \ldots & B_\k    &   B_{\k,\k}
    \end{pmatrix}=
 			    \begin{pmatrix}
 			       \ast &  \ast & \ldots & \ast & \ast  \\
 			       B_1 & \ast &  \ldots & \ast & \ast \\
                   \OO & B_{2}  & \ldots & \ast & \ast \\
                   \vdots & \vdots  & \ddots & \vdots & \vdots  \\
                   \OO & \OO & \ldots & B_{\k} & \ast 
   			 \end{pmatrix}
\end{equation}
{where every block $B_j$ is an $m_{j} \times m_{j-1}$ matrix of rank $m_j$ with $j = 1, 2, \ldots, \k$. Moreover, the $m_j$s are positive integers such that
\begin{eqnarray}
m_0 \geq m_1 \geq \ldots \geq m_\kappa \geq 1, \quad \textrm{and} \quad m_0+m_1+\ldots+m_\kappa=N.
\end{eqnarray}
We agree to let $m_0 :=m$ to have a consistent notation, moreover $\OO$ denotes a block matrix whose entries are zeros, whereas the coefficients of the blocks ``$\ast$'' are arbitrary. Note that we allow the operator $\L$ to be strongly degenerate, when $m < N$. However, the assumption {\rm \bf[H.1]} implies that the first order part $Y$ of $\L$ induces a strong regularity property. Indeed, it is known that $\L$ is hypoelliptic, namely that every distributional solution $u$ to $\L u = f$ defined in some open set $\Omega \subset \R^{N+1}$ belongs to $C^\infty(\Omega)$, and is a classical solution to $\L u = f$, whenever $f \in C^\infty(\Omega)$. In Section 2 we will recall several known results about the operators $\L$ that will be used in the sequel.  

In this article we study the local regularity of the classical solution $u$ to $\L u = f$ when $f$ is Dini continuous. For this reason we require as few conditions as possible for the definition of $\L u$. 

\begin{definition} \label{def-C2}
Let $\Omega$ be an open subset of $\R^{N+1}$. We say that a function $u$ belongs to $C^2_\L(\Omega)$ if $u$, its derivatives $\partial_{x_i}u, \partial_{x_i x_j}u$ ($i, j = 1, \dots, m$) and the Lie derivative $Yu$ defined in \eqref{eq-Yu} are continuous functions in $\Omega$. We also require, for $i=1, \dots, m$, that
\begin{equation} \label{eq-Ymix}
 \lim_{s \to 0} \frac{\partial_{x_i}u (\text{\rm exp} (s B )x, t-s) - \partial_{x_i}u(x,t)}{|s|^{1/2}} = 0,
\end{equation}
uniformly for every $(x,t) \in K$, where $K$ is a compact set $K \subset \Omega$. 

Let $f$ be a continuous function defined in $\Omega$. We say that a function $u$ is a classical solution to $\L u = f$ in $\Omega$ if $u$ belongs to $C^2_\L(\Omega)$, and the equation $\L u = f$ is satisfied at every point of $\Omega$. 
\end{definition}

Note that, as $\L$ is a linear second order differential operator, it is natural to consider $Y$ as a second order derivative, and \eqref{eq-Ymix} can be interpreted as a condition on the second order mixed derivative of the form $Y^{1/2} \partial_{x_i}u$. This condition will be used to prove that $u$ is approximated by its intrinsic Taylor polynomial of degree 2, which is one of the main achievements of this article. We recall that the \emph{nth}-order intrinsic Taylor polynomial of a function $u$ around the point $z$ can be defined as the unique polynomial function $P_z^n u$ of order $n$ such that 
\begin{eqnarray*}
u(\zeta)-P_z^n u(\zeta)=o(\Vert z^{-1} \circ \z \Vert_K^n) \quad \text{as} \quad \zeta \to z,
\end{eqnarray*}
where $\Vert \cdot \Vert_K$ denotes the semi-norm as defined in \eqref{def-norm} below.

\begin{theorem} \label{taylor} 
Let $\L$ be an operator in the form \eqref{e-Kolm-costc} satisfying hypothesis {\rm \bf[H.1]}. Let $\Omega$ be an open subset of $\R^{N+1}$ and let $u$ be a function in $C_{\L}^2(\Omega)$. For every $z:=(x,t) \in \Omega$ we define the second order Taylor polynomial of $u$ around $z$ as
\begin{equation}\label{def-tay}
\begin{split}
 T^2_{z}u(\z) & := u(z) + \sum_{i=1}^{m} \partial_{x_{i}} u(z) (\x_i-x_i) \\
& +\frac{1}{2} \sum_{i,j=1}^{m} \partial^2_{x_{i}x_{j}} u(z) (\x_i-x_i)(\x_j-x_j) - Yu(z)(\t-t),
\end{split}
\end{equation}
for any $\z = (\x,\t) \in \Omega$.
Moreover, we have
\begin{eqnarray}\label{tay-estimate}
u(\z)-T^2_{z}u(\z)=o(\Vert z^{-1} \circ \z \Vert_K^2) \quad \text{as} \quad \zeta \to z.
\end{eqnarray} 
\end{theorem}

Our main result is the local regularity of the classical solution $u$ to $\L u = f$ when $f$ is Dini continuous. In order to define a modulus of continuity which is suitable for the operator $\L$ we recall the Lie group structure $\mathbb{K}=(\R^{N+1},\circ)$ introduced by Lanconelli and Polidoro in \cite{LanconelliPolidoro}, and some related notation. In Section 2 we will explain its connection with $\L$. We let
\begin{equation} \label{eq-e}
	E(t) := \text{exp} (-t B ),
\end{equation}
and we define
\begin{equation}\label{e70}
\mathbb{K}=(\R^{N+1},\circ), \quad 
 (x,t)\circ (\x,\t)=(\x+ E(\t)x,t+\t), \quad (x,t),(\x,\t)\in \R^{N+1}.
\end{equation}
Then $\mathbb{K}$ is a non-commutative group with zero element $(0,0)$ and inverse
\begin{equation*}
(x,t)^{-1} = (-E(-t)x,-t).
\end{equation*}
In \cite{LanconelliPolidoro} it is proved that the operator $\L$ is invariant with respect to a family of dilations $(\delta_r)_{r>0}$ if, and only if, the matrix $B$ in \eqref{B} agrees with $B_0$ defined as:
 \begin{equation} \label{B_0}
 			   B_0 = \begin{pmatrix}
 			       \OO &  \OO & \ldots & \OO & \OO \\
 			       B_1 & \OO &  \ldots & \OO & \OO\\
                   \OO & B_{2}  & \ldots & \OO & \OO \\
                   \vdots & \vdots  & \ddots & \vdots & \vdots  \\
                   \OO & \OO & \ldots & B_{\k} & \OO
   			 \end{pmatrix}
\end{equation}
In other words, every block denoted by $*$ in \eqref{B} has zero entries. In this case the dilation is defined for every positive $r$ as
\begin{equation}\label{e-dilations}
    \delta_r :=\textrm{diag}(r I_{m}, r^3 I_{m_1}, \ldots, r^{2\kappa+1}I_{m_\kappa},r^2),
\end{equation}
where $I_k$, $k\in\N$, is the $k$-dimensional unit matrix. 

In the sequel we let $\L_0$ be the operator obtained from $\L$ by replacing its matrix $B$ with $B_0$ defined in \eqref{B_0}, and we base our blow-up argument on the family of dilations $(\delta_r)_{r>0}$. Hence we take advantage of the invariant structure of $\L_0$ in the study of the regularity of $\L$. This fact is quite natural as $\L_0$ is the \emph{blow-up limit of} $\L$, as it is shown in Section 2.2 of \cite{AnceschiPolidoro}. 

We now introduce a homogeneous semi-norm of degree 1 with respect to the family of dilations $(\delta_r)_{r >0}$ in \eqref{e-dilations} and a quasi-distance which is invariant with respect to the group operation in \eqref{e70}. We first rewrite the matrix $\delta_r$ with the equivalent notation
\begin{equation}\label{e-dilations-cp}
    \delta_r :=\textrm{diag}(r^{\alpha_1}, \ldots, r^{\alpha_N},r^2),
\end{equation}
where $\alpha_1, \dots, \alpha_{m_0} =1, \alpha_{m_0+1}, \dots, \alpha_{m_0+m_1} = 3, \alpha_{N-m_\kappa}, \dots, \alpha_N = 2 \kappa + 1$. 
\begin{definition}
	\label{norm-def}
    For every $(x,t) \in \R^{N+1}$ we set
  \begin{equation}\label{def-norm}
  \|(x,t)\|_K := \max \left\{ |x_1|^{\frac{1}{\alpha_1}}, \dots,  |x_N|^{\frac{1}{\alpha_N}}, |t|^{\frac{1}{2}} \right\}.
\end{equation}
\end{definition}
Note that the semi-norm is homogeneous of degree $1$ with respect to the family of dilations $(\delta_r)_{r>0}$, namely $ \|\delta_r (x,t) \|_K =r \| (x,t)\|_K $ for every $r>0$ and $(x,t)\in\R^{N+1}$. Moreover, the following pseudo-triangular inequality holds: for every bounded set $H \subset \R^{N+1}$ there exists a positive constant ${\bf c}_H$ such that
\begin{equation}\label{e-ps.tr.in}
 \|(x,t)^{-1}\|_K \le {\bf c}_H  \| (x,t) \|_K, \qquad  
 \|(x,t) \circ (\x,\t) \|_K \le {\bf c}_H  (\| (x,t) \|_K + \| (\x,\t) \|_K), 
\end{equation}
for every $(x,t), (\x,\t) \in H$. We then define the {\it quasi-distance} $d_K$ by setting
\begin{equation}\label{e-ps.dist}
    d_K( (x,t) ,(\x,\t)):= \|(\x,\t)^{-1}\circ (x,t)\|_K, \qquad (x,t), (\x,\t) \in \R^{N+1},
\end{equation}
and the {\it ball}
\begin{equation}\label{e-BK}
    \H_r (x_0,t_0):= \{ (x,t) \in \R^{N+1} \mid d_K((x,t), (x_0,t_0)) < r\}.
\end{equation}
Note that from \eqref{e-ps.tr.in} it directly follows
\begin{equation*}
    d_K((x,t) ,(\x,\t))\le {\bf c}_H (d_K((x,t) ,(y,s))+d_K((y,s) ,(\x,\t))), 
\end{equation*}
for every $ (x,t) ,(\x,\t),(y,s) \in \R^{N+1}$. We eventually define the modulus of continuity of a function $f$ defined on any set $H \subset \R^{N+1}$ as follows
\begin{equation} \label{eq-omegaf}
\omega_f(r) := \sup_{\substack{(x,t), (\x,\t)\in H \\ d_K((x,t) ,(\x,\t))<r}} |f(x,t)-f(\x, \t)|.
\end{equation}
\begin{definition} \label{d-DC}
A function $f$ is said to be Dini-continuous in $H$ if 
\begin{equation*}
\int_{0}^{1} \frac{\omega_f(r)}{r}dr < + \infty.
\end{equation*}
\end{definition}
We are now in position to state our main result.

\begin{theorem} \label{th-1}
Let $\L$ be an operator in the form \eqref{e-Kolm-costc} satisfying hypothesis {\rm \bf[H.1]}. Let $u\in C^2_\L(\H_{1}(0,0))$ be a classical solution to $\L u=f$. Suppose that $f$ is Dini continuous. Then there exists a positive constant $c$, only depending on the operator $\L$, such that:
\begin{description}
 \item[{\it i)}]
 \begin{equation*}
\vert \partial^2 u(0,0) \vert \le c \left(\sup_{\H_1(0,0)} \vert u \vert + \vert f(0,0) \vert +\int_{0}^1 \frac{\omega_{f}(r)}{r} \right);
\end{equation*}
 \item[{\it ii)}] for any points $(x,t)$ and $ (\x,\t) \in \H_{\frac14}(0,0)$ we have
\begin{equation*}
|\partial^2 u(x,t)- \partial^2 u(\x, \t)| \le c\left(d \sup_{\H_{1}(0,0)} |u| + d \sup_{\H_{1}(0,0)}|f|  + \int_0 ^d \frac{\omega_{f}(r)}{r}+ d \int_{d} ^1 \frac{\omega_{f}(r)}{r^2}\right).
\end{equation*}
where $d: =d_K ( (x,t), (\x,\t))$ and $\partial^2$ stands either for $\partial^2_{x_i x_j}$, with $i,j=1,\ldots,m$, or for $Y$.
\end{description}
\end{theorem}

We emphasize that Theorem \ref{th-1} fails even in the simplest Euclidean setting if we don't assume any regularity condition on the function $f$. Consider for instance the function 
\begin{equation*}
 u(x,y) = xy (\log(x^2+y^2))^\alpha, \qquad \text{with} \quad 0 < \alpha < 1.
\end{equation*}
A direct computation shows that 
\begin{equation*}
 \Delta u(x,y) = 8 \alpha \frac{xy}{x^2 + y^2} (\log(x^2+y^2))^{\alpha-1} + 4 \alpha(\alpha-1) \frac{xy}{x^2 + y^2} (\log(x^2+y^2))^{\alpha-2},
\end{equation*}
so that $f(x,y) := \Delta u(x,y)$ extends to a continuous function on $\R^2$, which is not Dini continuous at the point $(0,0)$. On the other hand, the derivative $\partial_x\partial_y u(x,y)$ is unbounded near the origin. We also point out that, when $\alpha =1$, the function $u$ is a counterexample for the $L^\infty$ bounds of the second order derivatives of weak solutions to $\Delta u= f$. \footnote{We acknowledge that this counterexample was pointed out to one of the authors by Andreas Minne during the Workshop ``New trends in PDEs'', held in Catania on 29-30 May 2018.}

We finally consider the non-constant coefficients operator $\LL$ defined as follows 
\begin{equation} \label{e-Kolm-var}
   \LL :=\sum_{i,j=1}^m a_{ij}(x,t) \partial^{2}_{x_i x_j}+\sum_{i,j=1}^N b_{ij}x_j\partial_{x_i}-\partial_t,
\end{equation}
We assume that the coefficients $a_{ij}$ are Dini continuous functions and, in order to simplify the notation, we write
\begin{equation} \label{eq-omega-a}
\omega_a(r) := \max_{i,j= 1, \dots, m}
\sup_{\substack{(x,t), (\x,\t)\in H \\ d_K((x,t) ,(\x,\t))<r}} |a_{ij}(x,t)-a_{ij}(\x, \t)|.
\end{equation}
We assume that the following condition on the matrix $A(x,t):=(a_{ij}(x,t))_{i,j=1,\ldots,m}$ is satisfied.

\medskip
\noindent
{\bf[H.2]} \ For every $(x,t) \in \R^{N+1}$, the matrix $A(x,t)$ is symmetric and satisfies
   \begin{equation}
 			   \label{A}
 			   \lambda |\xi|^2 \le \langle A(x,t) \xi, \xi \rangle \le \Lambda |\xi|^2, \qquad \text{for every} \quad \xi \in \R^m,
\end{equation}
for some positive constants $\lambda, \Lambda$. 

\medskip

\begin{theorem} \label{th-2}
Let $\LL$ be an operator in the form \eqref{e-Kolm-var} satisfying the hypotheses {\rm \bf[H.1]}  and {\rm \bf[H.2]}. Let $u\in C^2_\L(\H_{1}(0,0))$ be a classical solution to $\LL u=f$. Suppose that $f$ and the coefficients $a_{ij}$, $i,j= 1, \dots, m$, are Dini continuous. Then for any points $(x,t)$ and $ (\x,\t) \in \H_{\frac12}(0,0)$ the following holds:
\begin{equation*}
\begin{split}
|\partial^2 u(x,t)-\partial^2 u(\x, \t)| \le & c\Big( d \sup_{\H_{1}(0,0)} |u| + 
d \sup_{\H_{1}(0,0)}|f| + \int_0 ^d \frac{\omega_{f}(r)}{r}+ d \int_{d} ^1 \frac{\omega_{f}(r)}{r^2}\Big)  \\
+ & c \bigg( \sum_{i,j=1}^{m} \sup_{\H_{1}(0,0)} |\ \partial^2_{x_i x_j} u| \bigg)
\Big(\int_0 ^d \frac{\omega_{a}(r)}{r}+ d \int_{d} ^1 \frac{\omega_{a}(r)}{r^2}\Big).
\end{split}
\end{equation*}
where $d=d_K ( (x,t), (\x,\t))$ and $\partial^2$ stands either for $\partial^2_{x_i x_j}$,$i,j=1,\ldots,m$, or for $Y$.
\end{theorem}

We now compare our main findings with the current literature on this subject. We first consider functions $f$ defined on $H \subset \R^{N+1}$ that are H\"older continuous with respect to the distance \eqref{e-ps.dist}, that is
\begin{equation} \label{eq-Holder}
  |f(x,t)-f(\x, \t)| \le M \, d_K((x,t) ,(\x,\t))^\alpha,  \qquad \text{for every} \quad (x,t), (\x,\t) \in H,
\end{equation}
for some constants $M>0$ and $\alpha \in ]0, 1]$. In this case we say that $f \in C^{0,\alpha}_L(H)$ and we let
\begin{equation*} \label{eq-Holder-norm}
  \|f \|_{C^{0,\alpha}_L(H)} \:= \sup_{H} |f| + \inf \big\{ M \ge 0 \mid \eqref{eq-Holder} \ \text{holds} \big\}.
\end{equation*}
When $\alpha < 1$ we write $C^{\alpha}_L(H)$ instead of $C^{0,\alpha}_L(H)$.
As a direct consequence of Theorem \ref{th-2}  we have
\begin{corollary} \label{cor-1}
Let $u\in C^2_\L(\H_{1}(0,0))$ be a classical solution to $\LL u=f$. Suppose that $f$ and the coefficients $a_{ij}$, $i,j= 1, \dots, m$, belong to $C^{0,\alpha}_L(\H_1(0,0))$. Then for any points $(x,t)$ and $(\x,\t) \in \H_{\frac12}(0,0)$ the following holds:
\begin{equation*}
\begin{split}
 |\partial^2 u(x,t)-\partial^2 u(\x, \t)| \le & c \, d^\alpha \bigg(  \sup_{\H_{1}(0,0)} |u| + 
 \frac{\|f \|_{C^{\alpha}_L(\H_1(0,0))}}{\alpha(1-\alpha)} \\ 
 & + \sum_{i,j=1}^{m} \sup_{\H_{1}(0,0)} |\ \partial^2_{x_i x_j} u| 
 \frac{\|a \|_{C^{\alpha}_L(\H_1(0,0))}}{\alpha(1-\alpha)} \bigg) , 
 \qquad \text{if} \quad \alpha  <1,\\
|\partial^2 u(x,t)-\partial^2 u(\x, \t)| \le & c\, d \bigg( \sup_{\H_{1}(0,0)} |u| + 
\|f \|_{C^{0,1}_L(\H_1(0,0))} | \log \; d| \\
& + \bigg(\sum_{i,j=1}^{m} \sup_{\H_{1}(0,0)} |\ \partial^2_{x_i x_j} u| \bigg)
\|a \|_{C^{0,1}_L(\H_1(0,0))} | \log \; d| \bigg), \quad \text{if} \quad \alpha =1.
\end{split}
\end{equation*}
\end{corollary}
Note that, for $\alpha <1$, Corollary \ref{cor-1} restores the Schauder estimates previously proved by Manfredini in \cite{Manfredini}, and by Di Francesco and Polidoro in \cite{DiFrancescoPolidoro}. Note that, in this case, an interpolation inequality allows us to state a bound for the $C^{\alpha}_L$ norm of the derivatives $\partial^2 u$ in terms of $\|a \|_{C^{\alpha}_L (\H_{1}(0,0))}, \|f \|_{C^{\alpha}_L(\H_{1}(0,0))}$, and $\sup_{\H_{1}(0,0)} |u| $ only. We also recall that Schauder estimates in the framework of semigroups have been proved by Lunardi \cite{Lunardi}, Lorenzi \cite{Lorenzi}, Priola \cite{Priola}.  Theorems \ref{th-1} and \ref{th-2} improve the previous ones, not only because we weaken the regularity assumption on $f$ and on the coefficients $a_{ij}$'s, but also because the Schauder estimate for $\alpha = 1$ is not given in the aforementioned articles. We also quote analogous results obtained in the framework of stochastic theory (see Menozzi \cite{Menozzi} and its bibliography).
 
\medskip

The proof of our main results is based on the method introduced by Wang \cite{Wang} in the study of the Poisson equation, which greatly simplifies the other approaches previously used in literature. Wang considers in \cite{Wang} a solution $u$ to the equation $\Delta u = f$ in some open set $\Omega$. Without loss of generality, he assumes that the unit ball $B_1(0)$ is contained in $\Omega$ and considers a sequence of Dirichlet problems as follows. Let $B_{r_k}(0)$ be the Euclidean ball centered at the origin and of radius $r_k = \frac{1}{2^k}$, and let $u_k$ be the solution to the Dirichlet problem
\begin{equation*}
 \Delta u_k = f(0), \quad \text{in} \quad B_{r_k}(0), \qquad u_k = u \quad \text{in} \quad  \partial B_{r_k}(0).
\end{equation*}
Quantitative information on the derivatives of every solution $u_k$ is obtained by using only the elementary properties of the Laplace equation, namely the weak maximum principle, and the standard apriori estimates of the derivatives, that are obtained in \cite{Wang} via mean value formulas. The bounds for the derivatives of $u$ are obtained as the limit of the analogous bounds for $u_k$. The Taylor expansion in this step is crucial to conclude the proof. 

In this work we apply the method described above to degenerate Kolmogorov operators $\L$, by adapting Wang's approach to the non-Euclidean structure defined in \eqref{e70}. In particular, the ball $B_{r_k}(0)$ is replaced by the box $\H_{r_k}(0,0)$ defined through the dilation $\d_{r_k}$ introduced in \eqref{e-dilations}. Concerning the Taylor expansion, we recall the results due to Bonfiglioli \cite{Bonfiglioli} and the ones proved by Pagliarani, Pascucci and Pignotti \cite{PagliaPascPigno}. We emphasize that the authors of the above articles assume that the second order derivatives of the function $u$ are H\"older continuous, while we only require that $u$ belongs to the space $C^2_\L(\Omega)$ introduced in Definition \ref{def-C2}. As the regularity of the second order derivatives of $u$ is the very subject of this note, we don't assume extra conditions on them and we prove in Proposition \ref{taylor} the Taylor approximation under the minimal requirement that $u \in C^2_\L(\Omega)$.

\medskip

We conclude this introduction with a short discussion about the applicative and theoretical interest in the operator $\L$. A simple meaningful example is the operator introduced by Kolmogorov in \cite{Kolmogorov}, defined for $(x,t) = (v,y,t) \in \R^{m} \times \R^m \times \R$ as follows
\begin{equation} \label{e-Kolm-costc0}
   \K := \sum_{j=1}^m \partial^{2}_{x_j} - \sum_{j=1}^m x_j \partial_{x_{m+j}}-\partial_t 
   = \Delta_v - \langle v, D_y \rangle - \partial_t.
\end{equation} 
The operator $\K$ can be written in the form \eqref{e-Kolm-costc} with $\kappa = 1, m_1 = m$, and 
   \begin{equation} \label{B_K}
 			   B = \begin{pmatrix}
 			       \OO & \OO  \\ - I_m & \OO 
   			 \end{pmatrix}
\end{equation}
It arises in several areas of application of PDEs. In particular, in kinetic theory the density $u$ of particles, with velocity $v$ and position $y$ at time $t$, satisfies equation $\K u = 0$. In this setting, the Lie group has a quite natural intepretation. Indeed the composition law \eqref{e70} agrees with the \emph{Galilean} change of variables 
\begin{equation*}
    (v,y,t) \circ (v_{0}, y_{0}, t_{0}) = (v_{0} + v, y_{0} + y + t v_{0}, t_{0} + t), 
    \qquad (v,y,t), (v_{0}, y_{0}, t_{0}) \in \R^{2m+1}.
\end{equation*}
It is easy to see that $\K$ is invariant with respect to the above change of variables. Specifically, if $w(v,y,t) = u (v_{0} + v, y_{0} + y + t v_{0}, t_{0} + t)$ and $g(v,y,t) = f(v_{0} + v, y_{0} + y + t v_{0}, t_{0} + t)$, then 
\begin{equation*}
	\K u = f \quad \iff \quad \K w = g \quad \text{for  every} \quad (v_{0}, y_{0}, t_{0}) \in \R^{2m+1}.
\end{equation*}
As the matrix $B$ in \eqref{B_K} is in the form \eqref{B_0}, $\K$ is invariant with respect to the dilatation $\delta_r(v,y,t) := (r v, r^3 y, r^2 t)$. Note that the dilatation acts as the usual parabolic scaling with respect to the variable $v$ and $t$. The term $r^3$ in front of $y$ is due to the fact that the velocity $v$ is the derivative of the position $y$ with respect to time $t$. For a more exhaustive description of the operator $\L$, and of its applications, we refer to the survey article \cite{AnceschiPolidoro} by Anceschi and Polidoro and to its bibliography. 

After the work of Kolmogorov \cite{Kolmogorov} where \eqref{e-Kolm-costc0} was introduced, and H\"{o}rmander's celebrated article \cite{Hormander} on the hypoellipticity of second order degenerate linear operators, the regularity theory for operators that are invariant with respect to a Lie group structure has been widely developed by many authors. We quote here the seminal works by Folland \cite{Folland}, Folland and Stein \cite{FollandStein}, Rotschild and Stein \cite{RothschildStein}, Nagel, Stein and Wainger \cite{nagelstein}. We also refer to the monograph by Bonfiglioli, Lanconelli and Uguzzoni \cite{BLU} that contains an updated  description of this theory. Wei, Jiang, and Wu adapt in \cite{WeiJiangWu} the method introduced by Wang \cite{Wang} and prove Schauder estimates for hypoelliptic degenerate operators on the Heisenberg group. The Taylor formula used in \cite{WeiJiangWu} is proved by Arena, Caruso and Causa in \cite{ArenaCarusoCausa2010}. In a different framework, Wang's method has been used by Bucur and Karakhanyan \cite{BucurKarakhanyan} in the study of fractional operators. 

\medskip

This paper is structured as follows. In Section 2, we recall the basic facts about the analysis on Lie groups we need in our treatment. It also contains some recalls about the fundamental solution of the operator $\L$. In Section 3 we prove some preliminary results. In particular, we obtain some \emph{a priori} estimates of the derivatives of the solutions $u$ to $\L u = 0$ in terms of the $L^\infty$ norm of $u$. In Section 4 we prove our main result on the Taylor approximation of any function $u \in C^2_\L(\Omega)$. Section 5 contains the proof of Theorem \ref{th-1}, while Section 6 contains the proof of Theorem \ref{th-2}. 

\setcounter{equation}{0}\setcounter{theorem}{0}
\section{Lie Group Invariance and Fundamental Solution}
Here we discuss the invariance properties of Kolmogorov operators with respect to the Lie Group structure $\mathbb{K}=(\R^{N+1},\circ)$ introduced in \eqref{e70}. Moreover, we recall some known results concerning the fundamental solution of $\L$.

We first introduce some further notation. As the constant matrix $A$ is symmetric and positive, there exists a symmetric and positive matrix $A^{1/2} = \left( \overline{a}_{ij} \right)_{i,j= 1, \dots,m}$ such that $A = A^{1/2} A^{1/2}$. 
In order to check the hypothesis {\bf[H.1]}, we write $\L$ in terms of vector fields as follows
$$ \L=\sum_{i=1}^m X_i^2 +Y,$$
where
\begin{equation}\label{e-XiY}
 X_i := \sum_{j=1}^m \overline{a}_{ij} \partial_{x_j},\quad i=1,\ldots,m, \qquad Y := \langle B x, D \rangle -\partial_t,
\end{equation}
We recall that assumption {\bf [H.1]} is implied by H\"{o}rmander's condition (see \cite{Hormander}):
\begin{equation}\label{e-Horm}
{\rm rank\ Lie}\left(X_1,\dots,X_m,Y\right)(x,t) =N+1,\qquad \forall \, (x,t) \in \R^{N+1}.
\end{equation}
Yet another condition, equivalent to {\bf [H.1]}, (see \cite{LanconelliPolidoro}), is that
\begin{eqnarray*}
C(t) > 0, \quad \text{for every $t>0$},
\end{eqnarray*}
where
\begin{equation*} 	\label{c}
 C(t) = \int_{0}^{t} E(s) \, \begin{pmatrix}
 			       A & \OO  \\ \OO & \OO 
   			 \end{pmatrix} \, E^{T}(s) \, ds.
\end{equation*}

\medskip

We now recall that, under the the hypothesis of hypoellipticity, H\"{o}rmander constructed the fundamental solution of $\L$ as
\begin{equation*}
\Gamma(x,t,\x,\t)= \Gamma(x-E(t-\t)\x,t-\t),
\end{equation*}
where $\Gamma(x,t)=\Gamma(x,t,0,0)$ and $\Gamma(x,t)=0$ for every $ t \leq 0$, while
\begin{equation*}
\Gamma(x,t)= \frac{(4\pi)^{-\frac{N}{2}}}{\sqrt{\det C(t)}} \text{exp}\Big( -\frac{1}{4} \langle C^{-1}(t) x, x \rangle - t \, \text{tr} (B)\Big), \quad t > 0.
\end{equation*}
As a fundamental solution to $\L$, the following representation formula holds true: for every $u \in C_0^\infty(\rnn)$ we have
\begin{eqnarray}
u(z)=-\int_{\R^{N+1}}[\Gamma(z,\cdot)\L(u)](\z) d\z.
\end{eqnarray}
Here and in the sequel $z= (x,t)$ and $\z= (\x,\t)$ denote points of $\rnn$.

\medskip

We now conclude the analysis of the Lie Group $\mathbb{K}$, providing tools that will be very useful to prove our main results. We adopt the notation of  \cite{LanconelliPolidoro} and we quote the results therein. For a given $\z \in \R^{N+1}$, we denote by $\ell_{\z}$ the left translation on $\mathbb{K}=(\R^{N+1},\circ)$ defined as follows
\begin{equation*}
	\ell_{\z}: \R^{N+1} \rightarrow \R^{N+1}, \quad \ell_{\z} (z) = \z \circ z.
\end{equation*}
Then the vector fields  $X_1, \dots, X_m$ and $Y$ are left-invariant, with respect to the group law \eqref{e70}, in the sense that
\begin{equation}
	X_j\left( u (\zeta \circ \, \cdot \, ) \right) = \left( X_j u \right) (\zeta \circ \, \cdot \, ), \quad j=1, \dots, m, \qquad Y \left( u (\zeta \circ \, \cdot \, ) \right) = \left( Y u \right) (\zeta \circ \, \cdot \, ),
\end{equation}
for every $\zeta \in \rnn$ and every $u$ sufficiently smooth. Hence, in particular, 
\begin{eqnarray*}
\L \circ \ell_{\z}=\ell_{\z} \circ \L \quad \textrm{or, equivalently,} \quad \L\left( u (\zeta \circ \, \cdot \, ) \right) = \left( \L u \right) (\zeta \circ \, \cdot \, ).
\end{eqnarray*}

Regarding the invariance with respect to the dilation introduced in \eqref{e-dilations}, we recall that the operator $\L_0$, obtained from $\L$ by replacing its matrix $B$ with $B_0$ in \eqref{B_0}, satisfies
 \begin{equation}
	\label{Ginv}
      	 \L_0 \left( u \circ \delta_r \right) = r^2 \delta_r \left( \L_0 u \right), \quad \text{for every} \quad r>0,
\end{equation}
for every function $u$ sufficiently smooth (see Proposition 2.2 in \cite{LanconelliPolidoro}). In this case, we say that $\mathbb{K}= \left(\R^{N+1},\circ, \left( \delta_r \right)_{r>0}\right)$ is a \emph{homogeneous Lie group}, and we have
 \begin{equation*}
	\label{distr}
      	 \delta_r \left( z \circ \z \right) =  \left( \delta_r z \right) \circ \left( \delta_r \z \right), \quad \text{for every} \quad z, \zeta \in \rnn \ \text{and} \ r>0.
\end{equation*}

As we rely on a blow-up argument, we also apply the dilation \eqref{e-dilations} to the general operator $\L$ satisfying {\rm \bf[H.1]}. Specifically, we define $\L_r$ as the \textit{scaled operator} of $\L$ in terms of $(\delta_r)_{r>0}$ as follows
\begin{equation}\label{Lscaled}
\L_r :=r^2 (\delta_r \circ \L \circ \delta_{\frac{1}{r}}),
\end{equation}
and we write its explicit expression in terms of the matrix $B$ and $(\delta_r)$ as
\begin{equation}
\L_r =\sum_{i,j=1}^m a_{ij}\partial^2_{x_i x_j} + Y_r, \quad r \in (0,1]
\end{equation}
where
\begin{equation}
Y_r:=\langle B_r x, D \rangle - \partial_t
\end{equation}
and $B_r := r^2 \delta_r B \delta_{\frac{1}{r}}$, i.e.,
\begin{equation}
    \label{B_r}
    B_r = 
    \begin{pmatrix}
        r^2 B_{0,0} &   r^4 B_{0,1}  & \ldots &    r^{2\k} B_{0, \k - 1}  &  r^{2\k+2} B_{0, \k }  \\  
        B_1   &   r^2 B_{1,1}  & \ldots &    r^{2\k-2} B_{1, \k - 1}  &   r^{2\k} B_{1 , \k}  \\
        \OO    &    B_2  & \ldots &  r^{2\k-4} B_{2, \k - 1}   &  r^{2\k-2} B_{2 , \k}   \\
        \vdots & \vdots & \ddots & \vdots & \vdots \\
        \OO    &  \OO    &    \ldots & B_\k    &   r^2 B_{\k,\k}
    \end{pmatrix}.
\end{equation}
Clearly, $\L_r=\L$ for every $r >0$ if and only if $B=B_0$, and the principal part operator $\L_0$ is obtained as the limit of \eqref{Lscaled} as $r \to 0$.

Setting $E_r(t)=\exp(-tB_r)$, we define the translation group related to $\L_r$ as
\begin{equation}\label{circ_r}
(x,t) \circ_r (\xi,\tau)=(\xi + E_r(\tau)x, t+\tau),\quad (x,t), (\xi,\tau) \in \mathbb{R}^{N+1}.
\end{equation}
\begin{remark}\label{oss_circ}
As it will be useful in the blow-up limit procedure, we point out that the composition law defined in \eqref{circ_r} depends continuously on $r \in (0,1]$. Moreover, taking $r=0$ in \eqref{B_r} we find the matrix $B_0$ and ``$\; \circ_r \!\!$'' in \eqref{circ_r} simply becomes the composition law related to the dilation-invariant operator $\L_0$. Thus, ``$\; \circ_r \!\!$'' is a continuous function on the compact set $[0,1]$.
\end{remark}

The \emph{homogeneous dimension of $\R^{N+1}$ with respect to $(\delta_r)_{r >0}$} is the integer $Q+2$, where $Q$ is the so called \emph{spatial homogeneous dimension of $\R^{N+1}$}
\begin{equation}
	\label{hom-dim}
	Q := m + 3m_{1} + \ldots + (2\k+1) m_{\kappa}.
\end{equation}
We observe that the following equation holds true
\begin{equation*}
	 \det \delta_r = r^{Q + 2} \qquad \text{for every} \ r > 0.
\end{equation*}

\medskip

We now recall the notion of \emph{homogeneous function} in a homogeneous group. We say that a function $u$ defined on $\rnn$ is homogeneous of degree $\alpha \in \R$  if
\begin{equation*}
    u(\delta_{r}(z)) = r^{\alpha} u(z) \qquad \text{for every } \, z \in \R^{N+1}.
	\end{equation*}
A differential operator $X$ will be called homogeneous of degree $\b \in \R$ with respect to $(\delta_{r})_{r \ge 0}$ if 
\begin{equation*}
	X u (\delta_{r}(z)) = r^{\b} \left( X u \right) (\delta_{r}(z) ) \qquad \text{for every } \, z \in \R^{N+1},
\end{equation*}
and for every sufficiently smooth function $u$. Note that, if $u$ is homogeneous of degree $\alpha$ and $X$ is homogeneous of degree $\beta$, then $X u$ is homogeneous of degree $\alpha - \beta$. 
		
As far as we are concerned with the vector fields of the Kolmogorov operator $\L_0$ under the invariance assumption \eqref{Ginv}, we
have that $X_{1}, \ldots, X_{m}$ are homogeneous of degree $1$ and $Y$ is homogeneous of degree $2$ with respect to $(\delta_{r})_{r \ge 0}$. In particular, $\L_0 = \sum_{j=1}^m X_j^2 + Y$ is homogeneous of degree $2$, and its fundamental solution $\Gamma_0$ is a homogeneous function of degree $-Q$. As a direct consequence, the estimate $\Gamma_0(z,\zeta) \leq \frac{c}{\| \zeta^{-1}\circ z \|_K^Q}$ holds for every $z, \zeta \in \R^{N+1}$, with $z \ne \zeta$. Analogous bounds hold for the first order and second order derivatives of $\Gamma_0$, as they are homogeneous of degree $-Q-1$ and $-Q-2$, respectively.

In the sequel, as we also consider the non dilation-invariant operator $\L$, we rely on the following estimates (see Proposition 2.7 in \cite{DiFrancescoPolidoro}). Let $z_0 \in \R^{N+1}$ and $R_0 >0$ be a given point and a given constant. Assume that all the eigenvalues of the matrix $A$ belong to some interval $[\lambda, \Lambda] \subset \R^+$. Then there exists a positive constant $c$, only depending on $\lambda, \Lambda$ and on the matrix $B$, such that the following bounds hold
\begin{equation}\label{hom-1}
\begin{split}
 \Gamma(z,\zeta) & \leq \frac{c}{\| \zeta^{-1}\circ z \|_K^Q},\\ 
\vert \partial_{x_{j}} \Gamma(z,\zeta) &\vert \leq \frac{c}{\| \zeta^{-1}\circ z \|_K^{Q+1}},
\qquad \vert \partial_{\x_{j}} \Gamma(z,\zeta) \vert \leq \frac{c}{\| \zeta^{-1}\circ z \|_K^{Q+1}},\\
\vert \partial_{x_i x_{j}} \Gamma(z,\zeta) &\vert \leq \frac{c}{\| \zeta^{-1}\circ z \|_K^{Q+2}}, 
\qquad \vert \partial_{\x_i \x_{j}} \Gamma(z,\zeta) \vert \leq \frac{c}{\| \zeta^{-1}\circ z \|_K^{Q+2}},\\
\vert Y \Gamma(z,\zeta) &\vert \leq \frac{c}{\| \zeta^{-1}\circ z \|_K^{Q+2}}, 
\qquad \vert Y^* \Gamma(z,\zeta) \vert \leq \frac{c}{\| \zeta^{-1}\circ z \|_K^{Q+2}},
\end{split}
\end{equation}
for every $i,j=1, \dots, m$, $z, \zeta \in \H_{R_0}(z_0)$ with $z \ne \zeta$. Here $Y^*$ denotes the transposed operator of $Y$, defined as follows
\begin{eqnarray*}
\int_{\rnn} \phi(x,t) Y^{*} \psi(x,t) \, dx dt = \int_{\rnn}  \psi(x,t) Y \phi(x,t) \, dx dt,
\end{eqnarray*}
for every $\psi$, $\phi \in C^\infty_0(\mathbb{R}^{N+1})$.

A similar result holds for the derivatives $\partial_{x_{j}} \Gamma(z,\zeta)$ and $\partial_{\xi_{j}} \Gamma(z,\zeta)$ for $j = m+1, \dots, N$. These functions need to be considered as derivatives of order $\alpha_j$, where the integer $\alpha_j$ has been introduced in \eqref{e-dilations-cp}. We have 
\begin{equation}\label{hom-2}
\begin{split}
 \vert \partial_{x_{j}} \Gamma(z,\zeta) \vert \leq \frac{c}{\| \zeta^{-1}\circ z \|_K^{Q+\alpha_j}}, & \qquad
\vert \partial_{\xi_{k}} \Gamma(z,\zeta) \vert \leq \frac{c}{\| \zeta^{-1}\circ z \|_K^{Q+\alpha_k}}, \\
\vert \partial_{x_{j}} \partial_{\xi_{k}} \Gamma(z,\zeta) \vert \leq &\frac{c}{\| \zeta^{-1}\circ z \|_K^{Q+\alpha_j+ \alpha_k}},
\end{split}
\end{equation}
for every $j,k=1, \dots, N$, $z, \zeta \in \H_{R_0}(z_0)$ with $z \ne \zeta$. Note that, as $\alpha_1 = \dots = \alpha_m = 1$, the bounds in the first line of \eqref{hom-2} agree with the second line of \eqref{hom-1}. The proof of \eqref{hom-2} directly follows from the bound (2.59) and (2.60) in \cite{DiFrancescoPolidoro}.

We conclude this Section with the following corollary of the estimates \eqref{hom-1} and \eqref{hom-2}, which will be useful in the sequel.

\begin{lemma} \label{prop-ring}
 Assume that all the eigenvalues of the matrix $A$ belong to some interval $[\lambda, \Lambda] \subset \R^+$. Then there exist two positive constants $C$, only depending on $\lambda, \Lambda$ and on the matrix $B$, such that the following holds true. For every $R \in ]0,1]$ we have that 
\begin{equation}\label{sup-gamma}
\sup \left\{ \Gamma(z,\zeta) : z \in \H_{\frac{R}{2}}(0), \zeta \in \H_{R}(0) \setminus \H_{\frac{3R}{4}}(0) \right\} \le \frac{C}{R^Q}.
\end{equation}
 Moreover 
\begin{equation}\label{sup-gamma-derivatives}
\sup \left\{ \left|\partial_{x_{j}} \partial_{\xi_{k}} \Gamma(z,\zeta) \right| : z \in \H_{\frac{R}{2}}(0), \zeta \in \H_{R}(0) \setminus \H_{\frac{3R}{4}}(0) \right\} \le \frac{C}{R^{Q+\alpha_j+ \alpha_k}}.
\end{equation}
and
\begin{equation}\label{sup-gamma-derivatives-Y}
\sup \left\{ \left|Y \partial_{\xi_{k}} \Gamma(z,\zeta) \right| : z \in \H_{\frac{R}{2}}(0), \zeta \in \H_{R}(0) \setminus \H_{\frac{3R}{4}}(0) \right\} \le \frac{C}{R^{Q+2+ \alpha_k}}.
\end{equation}
for every $j,k=1, \dots, N$. 
\end{lemma}

\begin{proof}
We first choose $R_0 >0$ such that $\| \zeta^{-1}\circ z \|_K \le R_0$ whenever $z \in \H_{\frac{1}{2}}(0)$, and $\zeta \in \H_{1}(0)$. The existence of such a positive number follows from the pseudo-triangular inequality \eqref{e-ps.tr.in}. With this choice of $R_0$, we apply \eqref{hom-1}, and we find 
\begin{equation}\label{sup-gamma-1}
\begin{split}
 \sup & \left\{ \Gamma(z,\zeta) : z \in \H_{\frac{R}{2}}(0), \zeta \in \H_{R}(0) \setminus \H_{\frac{3R}{4}}(0) \right\} \le \\
& c \left( \inf \left\{ \| \zeta^{-1}\circ z \|_K : z \in \H_{\frac{R}{2}}(0), \zeta \in \H_{R}(0) \setminus \H_{\frac{3R}{4}}(0) 
\right\}\right)^{-Q}.
\end{split}
\end{equation}
We therefore need to estimate the infimum of $\Vert \zeta^{-1}\circ z \Vert_K$ for $z \in \H_{\frac{R}{2}}(0) $ and $\zeta \in  \H_{R}(0) \setminus \H_{\frac{3R}{4}}(0) $. 
We first consider the points $\bar{z}:=\delta_{\frac{1}{R}}(z) $ and $\bar{\zeta}:=\delta_{\frac{1}{R}}(\zeta) $ which belong to $ \H_{\frac{1}{2}}(0) $ and $\H_{1}(0) \setminus \H_{\frac{3}{4}}(0) $, respectively. 
We now define the function $ g(\bar{z},\bar{\zeta}):=\Vert \bar{\zeta}^{-1}\circ_R \bar{z} \Vert_K$, which is continuous on the compact set $E:=\overline{\H_{\frac{1}{2}}(0)} \times \overline{\H_{1}(0) \setminus \H_{\frac{3}{4}}(0)} \times [0,1] $, as observed in Remark \ref{oss_circ}. Thus, by Weierstrass's Theorem, $g$ attains a minimum $m$ on $E $, i.e.,
\begin{eqnarray*}
\Vert \bar{\zeta}^{-1}\circ_R \bar{z} \Vert_K \geq m , \quad \forall \bar{z} \in \overline{\H_{\frac{1}{2}}(0)}, \quad \forall \bar{\zeta} \in \overline{\H_{1}(0) \setminus \H_{\frac{3}{4}}(0)}, \quad \forall R \in [0,1].
\end{eqnarray*} 

Going back to the box of radius $R$, i.e. applying dilation $\delta_R$ to the points $\bar{z}$ and $\bar{\zeta}$ yields
\begin{eqnarray}\label{norm1}
\Vert \zeta^{-1}\circ z \Vert_K \geq m R, \quad z \in \overline{\H_{\frac{R}{2}}(0)}, \quad \zeta \in \overline{\H_{R}(0) \setminus \H_{\frac{3R}{4}}(0)},
\end{eqnarray} 
and therefore \eqref{sup-gamma-1} becomes
\begin{equation}
\begin{split}
 \sup  \left\{ \Gamma(z,\zeta) : z \in \H_{\frac{R}{2}}(0), \zeta \in \H_{R}(0) \setminus \H_{\frac{3R}{4}}(0) \right\} \le  \frac{c}{m^Q R^Q} = \colon \frac{C}{R^Q}
\end{split}
\end{equation}
where the constant $C$ does not depend on $R$.

To obtain \eqref{sup-gamma-derivatives} and \eqref{sup-gamma-derivatives-Y}, we use the bounds for the derivatives of $\Gamma$ in \eqref{hom-1} and apply the same arguments as above.
\end{proof}

\setcounter{equation}{0}\setcounter{theorem}{0}
\section{Preliminary results}

In this Section we list some preliminary facts, which are useful in proving our main results. First, we prove a priori estimates for the derivatives of $u$ solution to the Kolmogorov equation with right-hand side equal to $0$. To this end, we represent solutions to $\L u = 0$ as convolutions with the fundamental solution $\Gamma$  of $\L$ and its derivatives $\partial_{x_{1}}\Gamma,...,\partial_{x_{N}}\Gamma$.

We then prove a mean-value formula for $u$, which is based on the Euclidean mean-value theorem and on the homogeneity of the fundamental solution.


\medskip


\medskip

In order to state the first result of this Section, we recall the notation introduced in \eqref{e-dilations-cp}, that is $\delta_r =\textrm{diag}(r^{\alpha_1}, \ldots, r^{\alpha_N},r^2)$. In the sequel we assume that all the eigenvalues of the constant matrix $A$ belong to some interval $[\lambda, \Lambda] \subset \R^+$. We are now in position to state our result. 

\begin{proposition}\label{lem-apriori}
Let $u$ be a solution to $\L u = 0$ in $\H_{R}(z_0)$, with $R \in ]0,1]$. Then 
\begin{eqnarray*}
| \partial_{x_j}   u |(z) \leq \frac{C}{R^{\alpha_j}}\Vert u \Vert_{L^{\infty}(\H_{R}(z_0))},
\quad \text{for every} \quad z \in \H_{{\frac{R}{2}}}(z_0), \quad j=1,\ldots,N,
\end{eqnarray*}
for some positive constant $C$ only depending on $\lambda, \Lambda$ and on the matrix $B$.
\end{proposition}
\begin{proof}
Without loss of generality, we can assume $z_0=0$. Let $\eta_{R} \in C^{\infty}_{0}(\mathbb{R}^{N+1})$ be a cut-off function such that 
\begin{eqnarray}\label{cut-off}
\eta_{R}(x,t)=\chi(\Vert(x,t)\Vert_{K}),
\end{eqnarray}
where $\chi \in C^{\infty}([0,+\infty),[0,1])$ is such that $\chi(s)=1$ if $s \leq \frac{3R}{4}$, $\chi(s)=0$ if $s \geq R$ and $|\chi'| \leq \frac{c}{R}$, $|\chi''|\leq \frac{c}{R^2}$. Then, for every $z \in \H_{R}(0)$ and for $i=1,\ldots,N$, there exists a constant $c$, only depending on $B$, such that
\begin{eqnarray}\label{c_eta}
|\partial _{x_i} \eta_{R}(z)| \leq \frac{c}{R^{\alpha_i}}, \qquad |\partial_t \, \eta_R(z)| \leq \frac{c}{R^{2}}.
\end{eqnarray} 
Consequently, for every $z \in \H_R(0)$ and $i,j=1,\ldots,m$, we have $|\partial_{x_i x_j}^2 \eta_R (z)| \leq \frac{c}{R^2}$ and therefore we obtain a bound for the second order part of $\vert \L \eta_R(z) \vert$. 

Since $\eta_R \equiv 1$ in $\in \H_{\frac{3R}{4}}(0)$, for every $z \in \H_{\frac{R}{2}}(0)$ we represent a solution $u$ to $\L u=0$ as follows
\begin{align}\label{convolution}
u(z) = (\eta_{R}u)(z) = - \int_{\H_{R}(0)}[\Gamma(z,\cdot) \L(\eta_{R}u)](\z)d\z.
\end{align}
Since $\L=\div(AD_m)+Y$ and $\L u=0$ by assumption, \eqref{convolution} can be rewritten as
\begin{equation}\label{convolution2}
\begin{split}
u(z) = (\eta_{R}u)(z) &= - \int_{\H_{R}(0)}[\Gamma(z,\cdot) \div(A D_ m(\eta_{R})) u](\z)d \z\\
&\quad- \int_{\H_{R}(0)}[\Gamma(z,\cdot) Y(\eta_{R})u](\z)d\z \\
&\quad- 2 \int_{\H_{R}(0)}[\Gamma(z,\cdot)\langle D_m u, A D_m \eta_{R} \rangle ](\z)d\z.
\end{split}
\end{equation}
Integrating by parts the last integral in \eqref{convolution2}, we obtain, for every $z \in \H_{\frac{R}{2}}(0)$
\es{parts1}{
u(z) = (\eta_{R}u)(z) &= \int_{\H_{R}(0)}[\Gamma(z,\cdot) \div(A D_ m(\eta_{R})) u](\z)d \z\\ 
&\quad- \int_{\H_{R}(0)}[\Gamma(z,\cdot) Y(\eta_{R})u](\z)d\z \\
&\quad+ 2 \int_{\H_{R}(0)}[\langle D_m^\zeta \Gamma(z,\cdot), A D_m \eta_{R} \rangle u](\z)d\z,
}
where $D_m$ is the gradient with respect to $x_1,\ldots,x_m$ and the superscript in $D_m^\zeta $ indicates that we are differentiating w.r.t the variable $\zeta$. 

Since $z \in \H_{\frac{R}{2}}(0)$ and $\partial_{x_i}\eta_R,Y(\eta_R) = 0$ ($i=1,\ldots,m$) in $\H_{\frac{3R}{4}}(0)$, after differentiating under the integral sign \eqref{parts1}, we find
\en{
\partial_{x_{j}}u(z) = \partial_{x_{j}}(\eta_{R}u)(z) 
&= \int_{\H_{R}(0) \setminus \H_{\frac{3R}{4}}(0)}[\partial_{x_{j}} \Gamma(z,\cdot) \div(A D_ m(\eta_{R})) u](\z)d\z \\  
&\quad - \int_{\H_{R}(0) \setminus \H_{\frac{3R}{4}}(0)}[\partial_{x_{j}} \Gamma(z,\cdot) Y(\eta_{R}) u](\z)d\z \\ 
&\quad +2 \int_{\H_{R}(0) \setminus \H_{\frac{3R}{4}}(0)}[\langle \partial_{x_{j}} D_m^\zeta \Gamma(z,\cdot), A D_m \eta_{R}\rangle u] (\z)d\z, \\
}
for every $j=1,...,N$. Thus, we obtain

\en{
|\partial_{x_{j}}u(z)| = |\partial_{x_{j}}(\eta_{R}u)(z)| 
&\leq \int_{\H_{R}(0) \setminus \H_{\frac{3R}{4}}(0)}\big\vert[\partial_{x_{j}} \Gamma(z,\cdot)  \div(A D_ m(\eta_{R})) u](\z)\big\vert d\z \\ 
&\quad + \int_{\H_{R}(0) \setminus \H_{\frac{3R}{4}}(0)}\big\vert[\partial_{x_{j}} \Gamma(z,\cdot) Y(\eta_{R}) u](\z)\big\vert d\z \\   
&\quad +2 \int_{\H_{R}(0) \setminus \H_{\frac{3R}{4}}(0)}\big\vert[\langle \partial_{x_{j}} D_m^\zeta \Gamma(z,\cdot), A D_m \eta_{R}\rangle u] (\z)\big\vert d\z \\
& =: \tilde{I}_{1}(z) + \tilde{I}_{2}(z)+\tilde{I}_{3}(z),
}
We estimate $\tilde{I}_{1}(z)$ and $ \tilde{I}_{2}(z)$, for $z \in \H_{\frac{R}{2}}(0)$. We have
\begin{align*}
\tilde{I}_{1}(z) &\leq \Vert u \Vert_{L^{\infty}(\H_{R}(0))}\sup_{\H_{R}(0) \setminus \H_{\frac{3R}{4}}(0)}\vert \div(A D_ m(\eta_{R}))|\mis(\H_{R}(0)) \! \! \! \! \sup_{\substack{z \in \H_{\frac{R}{2}}(0),\\  \zeta \in \H_{R}(0) \setminus \H_{\frac{3R}{4}}(0)}}
\! \! \! \! \big\vert \partial_{x_{j}} \Gamma(z,\zeta)\big\vert ,\\
\tilde{I}_{2}(z) &\leq \Vert u \Vert_{L^{\infty}(\H_{R}(0))}\sup_{\H_{R}(0) \setminus \H_{\frac{3R}{4}}(0)}\vert Y(\eta_R)|\mis(\H_{R}(0)) \! \! \! \sup_{\substack{z \in \H_{\frac{R}{2}}(0),\\  \zeta \in \H_{R}(0) \setminus \H_{\frac{3R}{4}}(0)}} \! \! \! \big\vert \partial_{x_{j}} \Gamma(z,\zeta)\big\vert.
\end{align*}
We now apply Lemma \ref{prop-ring} and obtain
\begin{equation}\label{gamma_est}
\sup_{\substack{z \in \H_{\frac{R}{2}}(0),\\  \zeta \in \H_{R}(0) \setminus \H_{\frac{3R}{4}}(0)}}\big\vert \partial_{x_{j}} \Gamma(z,\zeta)\big\vert \leq \frac{\tilde{C}}{R^{Q+\alpha_j}}.
\end{equation}
Moreover, by our choice of the cut-off function $\eta_R$, we have 
\begin{equation}\label{div_cut}
 |\div(A D_m(\eta_R))| \leq \frac{\Lambda \, c}{R^2} \quad \text{in} \quad \H_R(0),
\end{equation}
where $\Lambda$ is the largest eigenvalue of $A$. Finally, combining inequalities \eqref{gamma_est} and \eqref{div_cut} with $\mis(\H_{R}(0)) = R^{Q+2}\mis(\H_{1}(0))$, we obtain
\begin{eqnarray}\label{ine-above-2}
\tilde{I}_{1}(z) \leq \frac{C}{R^{\alpha_j}}\Vert u \Vert_{L^{\infty}(\H_{R}(0))}, \quad z \in \H_{\frac{R}{2}}(0),
\end{eqnarray}
We now estimate $|Y(\eta_{R})|\leq |\langle B x, D\eta_{R}\rangle|+|\partial_t \eta_R|$ in $\H_{R}(0) \setminus \H_{\frac{3R}{4}}(0)$. The bound for the derivative with respect to time of $\eta_R$ is obtained using \eqref{c_eta}. 
Moreover
\begin{eqnarray}\label{sum2}
|\langle B x, D\eta_{R}(\zeta)\rangle| \leq \sum_{i,k=1}^N |b_{ik}||x_k||\partial_{x_i}\eta_R(\zeta)| 
\leq c \sum_{i,k=1}^N |b_{ik}| R^{\alpha_k-\alpha_i}, 
\end{eqnarray}
where $\zeta \in \H_{R}(0) \setminus \H_{\frac{3R}{4}}(0)$.
Notice that in sum \eqref{sum2} the exponent $\alpha_k-\alpha_i$ is always greater or equal to $-2$, because of the form of the matrix $B$. Since by assumption $R\le1$, we estimate \eqref{sum2} as follows
\begin{eqnarray}\label{add_term}
|\langle B x, D\eta_{R}\rangle| \leq \frac{C'}{R^2}, \quad \textrm{in $\H_{R}(0) \setminus \H_{\frac{3R}{4}}(0) $},
\end{eqnarray}
where $C'$ is a constant that only depends on the matrix $B$ and on the constant $c$ in \eqref{c_eta}.

Finally, using again $\mis(\H_{R}(0)) = R^{Q+2}\mis(\H_{1}(0))$, together with \eqref{gamma_est} and \eqref{add_term}, we obtain 
\begin{eqnarray}\label{ine-above}
\tilde{I}_{2}(z) \leq \frac{C}{R^{\alpha_j}}\Vert u \Vert_{L^{\infty}(\H_{R}(0))}, \quad z \in \H_{\frac{R}{2}}(0),
\end{eqnarray}
where $C$ depends only on the constants $c$ and $\tilde{C}$ in \eqref{c_eta} and \eqref{gamma_est} and on the matrix $B$.

By the same argument we prove that, for a point $z \in \H_{\frac{R}{2}}(0) $, we have
\begin{eqnarray*}
\tilde{I}_{3}(z) \leq \Vert u \Vert_{L^{\infty}(\H_{R}(0))}\frac{c}{R}\ \mis(\H_{R}(0)) \! \! \! \! \! \! \! \! 
\sup_{\substack{z \in \H_{\frac{R}{2}}(0),\\  \zeta \in \H_{R}(0) \setminus \H_{\frac{3R}{4}}(0)}} \! \! \! \! \! \!
\big\vert \partial_{x_{j}} D_m^\zeta \Gamma(z,\zeta)\big\vert \leq \frac{C}{R^{\alpha_j}}\Vert u \Vert_{L^{\infty}(\H_{R}(0))}.
\end{eqnarray*}
where $C$ denotes once again a constant depending only on $c$, $\tilde{C}$ and $B$. Combining the inequality above with \eqref{ine-above-2} and \eqref{ine-above}, we finally obtain 
\begin{equation*}
\Vert \partial_{x_{j}}u \Vert_{L^{\infty}(\H_{{\frac{R}{2}}}(0))} \leq \frac{C}{R^{\alpha_j}}\Vert u \Vert_{L^{\infty}(\H_{R}(0))}, \quad j=1,...,N.
\end{equation*}
\end{proof}

We state a result analogous to Proposition \ref{lem-apriori}, written in terms of the vector fields $X_1,\ldots,X_m,Y$ introduced in \eqref{e-XiY}.
\begin{proposition}\label{corollary}
Let $u$ be a solution to $\L u = 0$ in $\H_{R}(0)$, for $R \in ]0,1[$, then for any $X_i,X_j \in \lbrace X_1,...,X_m \rbrace$, there exists a constant $C$, only depending on $\lambda, \Lambda$ and on the matrix $B$, such that
\begin{eqnarray*}
| X_i  u |(z) \leq \frac{C}{R}\Vert u \Vert_{L^{\infty}(\H_{R}(0))},\quad z \in \H_{{\frac{R}{2}}}(0),\\
| X_i X_ j u |(z) \leq \frac{C}{R^2}\Vert u \Vert_{L^{\infty}(\H_{R}(0))},\quad z \in \H_{{\frac{R}{2}}}(0).
\end{eqnarray*}
Similarly, we have that
\begin{eqnarray*}\label{Y-estimate}
| Y u |(z) \leq \frac{C}{R^2}\Vert u \Vert_{L^{\infty}(\H_{R}(0))},\quad z \in \H_{{\frac{R}{2}}}(0).
\end{eqnarray*}
\end{proposition}
\pr{The estimate of $X_1, \dots, X_m$ has been proved in Proposition \ref{lem-apriori}. The proof of the remaining estimates is obtained by reasoning as in Proposition \ref{lem-apriori}, and using estimates \eqref{sup-gamma-derivatives} and \eqref{sup-gamma-derivatives-Y}, respectively. We omit the details here.}

In the sequel, we will need to estimate the second order derivatives of a solution to $\L u =g$, where $g$ is a polynomial of degree at most two. To this end, we let
\begin{equation}\label{def-g}
g_1(z) =\langle v, x \rangle,\qquad
g_2(z) =\langle M x, x \rangle,
\end{equation}
be two polynomial functions, where $v$ and $M$ denote a constant vector of $\R^N$ and a $N \times N$ constant matrix, respectively.

\begin{lemma}\label{lem-tec}
Let $\eta_R$ be the cut-off function introduced in \eqref{cut-off} and let $g_1$ and $g_2$ be the functions defined in \eqref{def-g}. Then there exists a positive constant $C$, only depending on $\lambda, \Lambda$ and on the matrix $B$, such that
\begin{eqnarray}\label{est-sec-gamma-const}
\bigg\vert  \partial^2_{x_i x_j}\int_{\H_R(0)}\Gamma(z,\z)\eta_R(\z)  d\z \bigg\vert \leq C,
\end{eqnarray}
\begin{equation}\label{est-sec-gamma}
\bigg\vert \partial^2_{x_i x_j}\int_{\H_R(0)}\Gamma(z,\z)\eta_R(\z) g_1(\z)  d\z \bigg\vert \leq C R,
\end{equation}
\begin{equation}\label{est-sec-gamma-2}
\bigg\vert \partial^2_{x_i x_j}\int_{\H_R(0)}\Gamma(z,\z)\eta_R(\z) g_2(\z)  d\z \bigg\vert \leq C R^2,
\end{equation}
for every $z \in \H_{\frac{R}{2}}(0)$, $R \in ]0,1]$ and for any $i,j=1,\ldots,m$.
\end{lemma}
\pr{
Reasoning as in the proof of Proposition 2.11 in \cite{DiFrancescoPolidoro}, we write the right-hand side of \eqref{est-sec-gamma-const} as 
\begin{equation}\label{2.111}
\begin{split}
\partial^2_{x_i x_{j}}\int_{\H_R(0)} \Gamma(\z^{-1}\circ z) \eta_R(\z) d\z &=\lim_{\varepsilon \to 0} \int_{\H_R(0) \cap \lbrace \Vert \z^{-1}\circ z \Vert_K \ge \varepsilon \rbrace}\partial^2_{x_i x_{j}}\Gamma(\z^{-1}\circ z) \eta_R(\z) d\z \\
&\quad +  \eta_R(z) \int_{\Vert \z \Vert_K=1} \partial_{x_i} \Gamma_0(\z)\nu_j d \sigma(\z)\\
&=:\lim_{\varepsilon \to 0}I_1^0(\varepsilon,z)+I_2^0(z).
\end{split}
\end{equation}
We rewrite $I_1^0(\varepsilon,z)$ as
\begin{equation}\label{I10}
\begin{split}
I^0_1(\varepsilon,z)&=\int_{\H_R(0) \cap \lbrace \Vert \z^{-1}\circ z \Vert_K \ge \varepsilon \rbrace}\partial^2_{x_i x_{j}}\Gamma(\z^{-1}\circ z) \left(\eta_R(\z)-\eta_R(z)\right) d\z \\
&\quad +\eta_R(z) \int_{\H_R(0) \cap \lbrace\Vert \z^{-1}\circ z \Vert_K \ge \varepsilon \rbrace}\partial^2_{x_i x_{j}}\Gamma(\z^{-1}\circ z) d\z. \\
\end{split}
\end{equation}
By the definition of $\eta_R$, we have
\begin{eqnarray}\label{bound-eta}
0 \leq \eta_R \leq 1, \quad \eta_R(\z)-\eta_R(z) =0, \quad \textrm{for any $\z \in \H_{\frac{3R}{4}}(0)$, $z \in \H_{\frac{R}{2}}(0)$}.
\end{eqnarray}
Thus, taking advantage of Lemma \ref{prop-ring}, we infer
\begin{equation}\label{est-const}
\begin{split}
&\bigg\vert \int_{\H_R(0) \cap \lbrace \Vert \z^{-1}\circ z \Vert_K \ge \varepsilon  \rbrace}\partial^2_{x_i x_{j}}\Gamma(\z^{-1}\circ z) \left(\eta_R(\z)-\eta_R(z)\right) d\z \bigg\vert\\
&\quad=\bigg\vert\int_{\H_R(0) \setminus \H_{\frac{3R}{4}(0)}}\partial^2_{x_i x_{j}}\Gamma(\z^{-1}\circ z) \left(\eta_R(\z)-\eta_R(z)\right) d\z \bigg\vert \leq \frac{C}{R^{Q+2}}R^{Q+2}=C
\end{split}
\end{equation}
Thus we find
\begin{equation}\label{est-const2}
\begin{split}
&I^0_2(z)+ \lim_{\varepsilon \to 0}\eta_R(z) \int_{\H_R(0) \cap \lbrace \Vert \z^{-1}\circ z \Vert_K \ge \varepsilon \rbrace}\partial^2_{x_i x_{j}}\Gamma(\z^{-1}\circ z) d\z=C.
\end{split}
\end{equation}
Combining estimates \eqref{est-const} and \eqref{est-const2} we conclude the proof of \eqref{est-sec-gamma-const}.

We now prove \eqref{est-sec-gamma}. Reasoning as in \eqref{2.111} and exploiting the definition of $g_1$, we can rewrite the right-hand side of \eqref{est-sec-gamma} as 
\begin{equation}\label{2.11}
\begin{split}
\partial^2_{x_i x_{j}}\int_{\H_R(0)} \!\!\! \!\!\! \Gamma(\z^{-1}\circ z) \eta_R(\z)\langle v,\xi \rangle d\z &=\lim_{\varepsilon \to 0} \int_{\H_R(0) \cap \lbrace \Vert \z^{-1}\circ z \Vert_K \ge \varepsilon \rbrace}
\!\!\! \!\!\!  \partial^2_{x_i x_{j}}\Gamma(\z^{-1}\circ z) \eta_R(\z)\langle v,\xi \rangle d\z \\
&\quad + \langle v, x \rangle \eta_R(z) \int_{\Vert \z \Vert_K=1} \partial_{x_i} \Gamma_0(\z)\nu_j d \sigma(\z)\\
&=:\lim_{\varepsilon \to 0}I_1^1(\varepsilon,z)+I_2^1(z).
\end{split}
\end{equation}
We prove that the first integral in \eqref{2.11} uniformly converges as $\varepsilon \to 0^+$. We first rewrite $I_1^1(\varepsilon,z)$ as
\begin{equation}
\begin{split}
I_1^1(\varepsilon,z)&=\int_{\H_R(0) \cap \lbrace \Vert \z^{-1}\circ z \Vert_K  \ge   \varepsilon \rbrace}\partial^2_{x_i x_{j}}\Gamma(\z^{-1}\circ z) \left(\eta_R(\z)-\eta_R(z)\right)\langle v,\xi \rangle d\z \\
&\quad +\int_{\H_R(0) \cap \lbrace \Vert \z^{-1}\circ z \Vert_K \ge \varepsilon \rbrace}\partial^2_{x_i x_{j}}\Gamma(\z^{-1}\circ z) \left(\eta_R(z)\right)\langle v,\xi -x \rangle d\z \\ 
&\quad +\langle v, x \rangle \eta_R(z) \int_{\H_R(0) \cap \lbrace \Vert \z^{-1}\circ z \Vert_K  \ge   \varepsilon \rbrace}\partial^2_{x_i x_{j}}\Gamma(\z^{-1}\circ z) d\z \\
&=: I'_1(\varepsilon,z)+I'_2(\varepsilon, z)+I'_3(\varepsilon,z).
\end{split}
\end{equation}
%
To estimate $I'_1(\varepsilon,z)$ we use the same argument as in \eqref{est-const}, with the only difference that now in the integral we have the additional term $\langle v, \xi \rangle$. We find a bound for this term observing that

\begin{equation}\label{scal-prod-I1}
|\langle v, \xi \rangle | \leq \Vert v \Vert \cdot \Vert \z \Vert_K \leq \Vert v \Vert \cdot R,
\end{equation}
where $\Vert v \Vert$ denotes the norm of $v$ in $\rn$.
Therefore, we obtain
\begin{eqnarray}\label{lim-2}
|I'_1(\varepsilon,z)| \leq C \int_{\H_R(0) \cap \lbrace \Vert \z^{-1}\circ z \Vert_K  \ge   \varepsilon \rbrace}\frac{d\z}{\Vert \z^{-1}\circ z \Vert_K^{Q+1}}\leq c \, R,
\end{eqnarray}
where $C$ is a constant that depends only on $\lambda, \Lambda, B$ and $v$.

We now show that the same bound holds for $I'_2(\varepsilon,z)$. We first observe that
\begin{eqnarray}
|\langle v, x-\xi \rangle| \leq \Vert v \Vert \cdot \Vert \z^{-1} \circ z \Vert_K,
\end{eqnarray}
As a consequence, using again \eqref{bound-eta} and \eqref{hom-1}, we infer
\begin{eqnarray}\label{lim-1}
|I'_2(\varepsilon,z)| \leq  \int_{\H_R(0) \cap \lbrace \Vert \z^{-1}\circ z \Vert_K  \ge   \varepsilon \rbrace}\frac{d\z}{\Vert \z^{-1}\circ z \Vert_K^{Q+1}}\leq c (R -\varepsilon) \le c \, R.
\end{eqnarray}
Using \eqref{lim-2} and \eqref{lim-1} we obtain
\begin{eqnarray}
\lim_{\varepsilon \to 0^+}I'_1(\varepsilon,z) =O(R), \qquad \lim_{\varepsilon \to 0^+}I'_2(\varepsilon,z) =O(R),\quad \textrm{as $R \to 0$}.
\end{eqnarray} 
Finally, as for $I'_3(\varepsilon,z)$, we compute
\begin{equation*}
\begin{split}
&\lim_{\varepsilon \to 0}\int_{\H_R(0) \cap \lbrace  \varepsilon \leq \Vert \z^{-1}\circ z \Vert_K \leq \textbf{c}R \rbrace}\partial^2_{x_i x_{j}}\Gamma(\z^{-1}\circ z)d\z \\
&=\lim_{\varepsilon \to 0}\int_{\H_R(0) \cap \lbrace  \varepsilon \leq \Vert \z^{-1}\circ z \Vert_K \leq \textbf{c}R \rbrace}\partial^2_{w_i w_{j}}\Gamma(w) e^{-\tau \, {\rm tr} B}dw\\
&=-\lim_{\varepsilon \to 0} \int_{\Vert w\Vert_K =\varepsilon} \partial_{w_i}\Gamma(w)e^{-\tau \, {\rm tr} B} \nu_j d\sigma_j(w)+\lim_{\varepsilon \to 0} \int_{\Vert w\Vert_K =\textbf{c}R} \partial_{w_i}\Gamma(w)e^{-\tau \, {\rm tr} B} \nu_j d\sigma_j(w) \\
&=-\int_{\Vert w \Vert_K=1}\partial_{w_i}\Gamma_0(w)\nu_j d\sigma_j(w)+\int_{\Vert w\Vert_K =\textbf{c}R} \partial_{w_i}\Gamma(w)e^{-\tau \, {\rm tr} B} \nu_j d\sigma_j(w).
\end{split}
\end{equation*}
We then obtain, 
\begin{equation}
\begin{split}
I_2^1(z)+ \lim_{\varepsilon \to 0}I'_3(\varepsilon, z) &= \langle v, x \rangle \eta_R(z) \int_{\Vert \z \Vert_K=1} \partial_{x_i} \Gamma_0(\z)\nu_j d \sigma(\z)\\
&\quad-\langle v, x \rangle \eta_R(z)\int_{\Vert w \Vert_K=1}\partial_{w_i}\Gamma_0(w)\nu_j d\sigma_j(w)\\
&\quad+\langle v, x \rangle \eta_R(z)\int_{\Vert w\Vert_K =\textbf{c}R} \partial_{w_i}\Gamma(w)e^{-\tau \, {\rm tr} B} \nu_j d\sigma_j(w)\\
&=\langle v, x \rangle \eta_R(z)\int_{\Vert w\Vert_K =\textbf{c}R} \partial_{w_i}\Gamma(w)e^{-\tau \, {\rm tr} B} \nu_j d\sigma_j(w).
\end{split}
\end{equation}
Keeping in mind that
\begin{eqnarray*}
\lim_{R\to 0}\int_{\Vert w\Vert_K =\textbf{c}R} \partial_{w_i}\Gamma(w)e^{-\tau \, {\rm tr} B} \nu_j d\sigma_j(w)=\int_{\Vert w \Vert_K=1}\partial_{w_i}\Gamma_0(w)\nu_j d\sigma_j(w)=c',
\end{eqnarray*}
we finally find 
\begin{equation}\label{lim-3}
I_2^1(z)+ \lim_{\varepsilon \to 0}I'_3(\varepsilon, z)=O(R),\quad \textrm{as $R\to 0$}.
\end{equation}
Identity \eqref{est-sec-gamma} follows from \eqref{lim-2}, \eqref{lim-1} and \eqref{lim-3}.

By the same argument, we obtain
\begin{equation}\label{der-seconde2}
\partial^2_{x_i x_{j}}\int_{\H_R(0)} \Gamma(\z^{-1}\circ z) \eta_R(\z)\langle M \xi,\xi \rangle d\z = O(R^2).
\end{equation}
We omit the details here as the procedure is analogous.
}

From Proposition \ref{lem-apriori} and Lemma \ref{lem-tec}, we derive the following result.
\begin{lemma}\label{lem-tec-2}
Let $w$ be a solution to $\L w = \langle v, \xi \rangle + \langle M \xi ,\xi \rangle$ in $\H_R(0)$, where $v$ and $M$ are as in 
\eqref{def-g}. Then
\begin{eqnarray}\label{est-lem-tec-2}
|\partial^2_{x_i x_j} w(z)| \leq \frac{C}{R^2}\Vert w \Vert_{L^\infty(\H_R(0))} + C R,
\end{eqnarray}
for every $i,j=1,\ldots,m, 0 < R \le 1$ and for any $z \in \H_{\frac{R}{2}(0)}$.
\end{lemma}
\pr{
Reasoning as in Proposition \ref{lem-apriori}, we obtain
\begin{equation}\label{overlineI_2}
\begin{split}
\vert \partial^2_{x_i x_{j}}w(z)\vert 
&\leq \int_{\H_R(0) \setminus \H_{\frac{3R}{4}}(0)}\big\vert[\partial^2_{x_i x_{j}} \Gamma(z,\cdot) \div(A D_ m(\eta_{R})) w ](\z)\big\vert d\z \\ 
&\quad + \int_{\H_R(0) \setminus \H_{(3/4) \varrho^k}(0)}\big\vert[\partial^2_{x_i x_{j}} \Gamma(z,\cdot) Y(\eta_{R}) w ](\z)\big\vert d\z \\    
&\quad +2 \int_{\H_R(0) \setminus \H_{\frac{3R}{4}}(0)}\big\vert[\partial^2_{x_i x_{j}} \Gamma(z,\cdot) \langle  D_m w, A D_m \eta_R\rangle ] (\z)\big\vert d\z \\
&\quad +\bigg\vert\Big[\partial^2_{x_i x_{j}}\int_{\H_R(0)}[ \Gamma(z,\cdot) \eta_R ](\z)\langle v,\xi \rangle  d\z \Big]\bigg\vert\\
&\quad +\bigg\vert\Big[\partial^2_{x_i x_{j}}\int_{\H_R(0)}[ \Gamma(z,\cdot) \eta_R ](\z)\langle M \xi,\xi \rangle  d\z \Big]\bigg\vert\\
& =: \overline{I}_{1}(z) + \overline{I}_{2}(z)+\overline{I}_3(z)+\overline{I}_{4}(z)+\overline{I}_{5}(z).
\end{split}
\end{equation}
The terms $\overline{I}_{1}(z), \overline{I}_{2}(z), \overline{I}_3(z)$ were already estimated in Proposition \ref{lem-apriori} as
\begin{eqnarray*}
\overline{I}_{1}(z),\overline{I}_{2}(z), \overline{I}_{3}(z) \leq \frac{C}{R^2}\Vert w \Vert_{L^\infty(\H_R(0))}, \quad z \in \H_{\frac{R}{2}(0)}.
\end{eqnarray*}
Additionally, $\overline{I}_{4}(z)$ and $\overline{I}_{5}(z)$ are $O(R)$ in virtue of Lemma \ref{lem-tec} and thus \eqref{est-lem-tec-2} is proved.
}
}

We now prove a mean value theorem for solutions $u$ to $\L u=0$ in cylinders $\H_R(\z)$.
\begin{proposition} [Mean value theorem] \label{mean-value-lem}
Let $\z$ be any point of $\R^{N+1}$, and let $u$ be a solution to $\L u=0$ in $\H_R(\z)$, with $R \in ]0,1]$. Then the following estimate holds
\begin{eqnarray}\label{lag}
\vert u (z)- u (\z) \vert \leq \frac{C}{R}d_K(z,\z) \Vert u \Vert_{L^{\infty}(\H_{R}(\z))},
\end{eqnarray}
for every $z \in \H_\frac{R}{2}(\z)$. Here $C$ is a constant that only depends on $\lambda, \Lambda$ and on the matrix $B$.
\end{proposition}
\pr{
Thanks to the left-invariance of the operator $\L$, it is not restrictive to assume $\z=0$, then we need to prove 
\begin{eqnarray*}
\vert u (z)- u (0) \vert \leq \frac{C}{R}\Vert z \Vert_K \Vert u \Vert_{L^{\infty}(\H_{R}(0))}.
\end{eqnarray*}
Consider $z=(x,t) \in \H_{\frac{R}{2}}(0)$, and apply the standard mean-value theorem
\begin{equation}\label{sum}
\begin{split}
|u(z)-u(0)| &= |u(x_1,\ldots,x_N,t)-u(0,\ldots,0,0)| \\ 
&\leq \sum_{i=1}^{N} |x_i| \, |\partial_{x_i} u(\vartheta_1 x_1,\ldots,\vartheta_N x_N,t)| + |t| \,|Yu(0,\ldots,0,\vartheta t),
\end{split}
\end{equation}
where $\vartheta_1,\ldots,\vartheta_N,\vartheta \in ]0,1[$. For every $i=1,\ldots,N$, we have $|x_i| \leq \Vert z \Vert_K^{\alpha_i} \leq R^{\alpha_i}$,
and $(\vartheta_1 x_1,\ldots,\vartheta_N x_N,t) \in \H_{\frac{R}{2}}(0)$. Then, by Proposition \ref{lem-apriori}, we find
\begin{eqnarray*}
|\partial_{x_i} u (\vartheta_1 x_1,\ldots,\vartheta_N x_N,t)| \leq \frac{c}{R^{\alpha_i}} \Vert u \Vert_{L^{\infty}(\H_{R}(0))}.
\end{eqnarray*}
so that
\begin{eqnarray*}
|x_i| \, |\partial_{x_i} u (\vartheta_1 x_1,\ldots,\vartheta_N x_N,t)| \leq 
\frac{c}{R} \Vert z \Vert_K \Vert u \Vert_{L^{\infty}(\H_{R}(0))}.
\end{eqnarray*}
Analogously, we have that $|\vartheta t| \le | t| \le \Vert z \Vert_K^{2} \leq R^2$, and from Proposition \ref{corollary} it follows that
\begin{eqnarray*}
|Y u (0,\dots, 0,\vartheta t)| \leq \frac{c}{R^{2}} \Vert u \Vert_{L^{\infty}(\H_{R}(0))},
\end{eqnarray*}
thus
\begin{eqnarray*}
|t| \, |Y u (0,\dots, 0,\vartheta t)| \leq \frac{c}{R} \Vert z \Vert_K \Vert u \Vert_{L^{\infty}(\H_{R}(0))}.
\end{eqnarray*}
The proof of the proposition can be obtained by combining the above estimates.
}

\setcounter{equation}{0}\setcounter{theorem}{0}
\section{Taylor formula}
In this Section we prove Theorem \ref{taylor}. The proof is based on the method introduced by Pagliarani, Pascucci and Pignotti in \cite{PagliaPascPigno} for the dilation-invariant operator $\L_0$ and then generalized by Pagliarani and Pignotti in \cite{PagliaPigno} to the non dilation-invariant one. 
In both articles the function $u$ is assumed to belong to some H\"older space $C^{2+ \alpha}_\L(\Omega)$ and the following bound 
\begin{eqnarray}\label{tayH-estimate}
u(\z)-T^2_{z}u(\z)=O(\Vert z^{-1} \circ \z \Vert_K^{2+ \alpha}) \quad \text{as} \quad \zeta \to z
\end{eqnarray} 
is proved instead of \eqref{tay-estimate}. Note that the case $\alpha = 0$ is covered in \eqref{tayH-estimate}, but \eqref{tay-estimate} contains a stronger statement. We follow the same procedure introduced in \cite{PagliaPascPigno} and \cite{PagliaPigno} and we point out the modifications needed to deal with our slightly different situation. 

We next introduce some further notation. We define the spaces $V_0, \dots, V_\kappa$ as the vector subspaces of $\mathbb{R}^N$ which are invariant with respect to dilation $(\delta_r)_{r>0}$ introduced in \eqref{e-dilations}. Specifically, for $n=0,\ldots,\kappa$, we set
\begin{equation*}
 V_n := \lbrace 0 \rbrace^{\bar{m}_{n-1}} \times \mathbb{R}^{m_n} \times \lbrace 0 \rbrace^{N- \bar{m}_{n}},
\end{equation*}
where $\bar{m}_{n}:=m_0+\ldots m_n$, with $m_{-1}\equiv0$. Moreover, we let $x^{[n]}$ be the projection of $x \in \mathbb{R}^N$ on $V_n$. Note that
\begin{eqnarray} \label{eq-sommadiretta}
\mathbb{R}^N = \bigoplus_{n=0}^{\kappa} V_n, \qquad x = x^{[0]} + \dots + x^{[\kappa]},
\end{eqnarray}
for every $x \in \R^N$. Moreover, in accordance with the dilation $(\delta_r)_{r>0}$, we have
\begin{eqnarray}
\delta_r (x^{[n]})=r^{2n+1}x^{[n]}, \quad \forall x^{[n]} \in V_n,
\end{eqnarray}
for every $n=0,\ldots,\kappa$. In virtue of assumption {\bf[H.1]}, the linear application $B^n : V_0 \rightarrow V_n$ is surjective; however, it is in general not injective.
Thus, we define the subspaces $V_{0,n} \subset V_0$ as follows
\begin{eqnarray*}
V_{0,n}:= \ker(B^n)^{\perp}.
\end{eqnarray*}
The linear map $B^n : V_{0,n} \rightarrow V_n$ is now bijective.

\medskip


The method of the proof relies on the construction of a finite sequence of points which connect $z = (x,t)$ and $\z=(\xi, \tau)$ and are located along suitable trajectories. More precisely, we start from $z$ and choose $z_1 = (x_1, t_1)$ as the point along the integral curve of the drift $Y$ satisfying the condition $t_1 = \tau$. We then move along the integral paths of $X_1,\ldots,X_m$ to a point $z_2 = (x_2, t_2)$ such that $x_2^{[0]}= \xi^{[0]}$ and $t_2=\tau $. This allows us to exploit the regularity of $u$ along the vector fields $X_1,\ldots,X_m, Y $ and estimate the remainder in \eqref{tay-estimate} in terms of the homogeneous norm of the new points. 

Since we have no apriori regularity of $u$ with respect to other vector fields, we increment the higher level coordinates $x^{[1]}, \dots, x^{[\kappa]}$ by moving along trajectories defined as concatenations of integral curves of $X_1,\ldots,X_m,Y$. 
Specifically, for any $z \in \mathbb{R}^{N+1}$ and $s \in \mathbb{R}$ we define iteratively the family of trajectories $(\gamma_{v,s}^{(n)}(z))_{n=0,\ldots,\kappa}$ as follows
\es{traj}{
\gamma_{v,s}^{(0)}(z)&=e^{s X_v}(z)=(x+s v,t)\\
\gamma_{v,s}^{(n+1)}(z)&=e^{-s^2Y}(\gamma_{v,-s}^{(n)}(e^{s^2Y}(\gamma_{v,s}^{(n)}(z)))),
}
where $v$ is a suitable vector in $V_0$, and $X_v = v_1 \partial_{x_1} + \dots + v_m \partial_{x_m}$.

At this point we need to distinguish the dilation-invariant operators from the non dilation-inviariant ones. In the first case, the trajectories $(\gamma_{v,s}^{(n)}(z))_{n=0,\ldots,\kappa}$ have the remarkable property of modifying the components $x^{[n]}+\dots + x^{[\kappa]}$ leaving unchanged the components $x^{[0]}+\dots + x^{[n-1]}$; thus, we reach the point $\zeta$ after $\kappa$ steps. The proof of Theorem \ref{taylor}, for dilation-invariant operators, follows by exploiting the regularity of $u$ with respect to $X_1, \dots, X_m, Y$, as we connect $z$ to $\zeta$ along integral curves of the vector fields $X_1, \dots, X_m, Y$. The next example illustrates the geometric construction in the simplest case, corresponding to $\kappa = 1$.

\begin{example}\label{simple_ex}
We consider the degenerate Kolmogorov operator 
\begin{equation*}
\mathcal{K}_0:= \partial^2_{x x} + x \partial_y -\partial_t
\end{equation*}
and show how to use the trajectories defined in \eqref{traj} to connect an arbitrary point $z \in \mathbb{R}^3$ with the origin.
In this case, we have
\begin{equation*}
 			   B = \begin{pmatrix}
 			       0 &  0   \\
 			       1 & 0 
   			 \end{pmatrix}
\end{equation*}
and thus
\begin{eqnarray*}
e^{s X}(x,y,t)=(x+s,y,t), \quad e^{s Y}(x,y,t)=(x,y+sx,t-s).
\end{eqnarray*}
Moreover, 
\begin{eqnarray*}
\mathbb{R}^2 =V_0 \oplus V_1 = \textrm{\rm{span}} \lbrace e_1 \rbrace \oplus \textrm{\rm{span}} \lbrace e_2 \rbrace, \quad V_{0,0}=V_{0,1} = \textrm{\rm{span}} \lbrace e_1 \rbrace.
\end{eqnarray*}
Let $z=(x,y,t)$ be a point in $\mathbb{R}^3$, and consider for simplicity $\zeta = (0,0,0)$. We first adjust the temporal component by moving along the drift $Y$, and we reach the point
\begin{eqnarray*}
z_1=e^{t Y}(z)=(x,y+t x,0).
\end{eqnarray*}
We then move along the integral curve of the vector field $X$ to make $x$ equal to $0$:
\begin{eqnarray*}
z_2=e^{s_0 X}(z_1)=(x+s_0, y+tx,0)=(0,y+tx,0),  \quad \text{by choosing} \quad s_0 = -x.
\end{eqnarray*}
We reached the point $z_2 \in V_1$ and we plan to steer it to $(0,0,0)$. We move along a curve defined as concatenation of integral paths of $X$ and $Y$ as follows:
\es{ex-traj}{
z_3&=e^{s_1 X}(z_2)=(s_1,y+tx,0),\\
z_4&=e^{s_1^2 Y}(z_3)=(s_1,y+tx + s_1^3 ,-s_1^2),\\
z_5&=e^{-s_1 X}(z_4)=(0,y+tx + s_1^3 ,-s_1^2),\\
z_6&=e^{-s_1^2 Y}(z_5) =(0,y+tx + s_1^3 ,0),
}
and we reach the point $\zeta = (0,0,0)$ if we choose $s_1 = (-tx-y)^{\frac{1}{3}}$.
\end{example}

When considering a non dilation-invariant operator $\L$, the method illustrated above fails. Indeed, in this case the trajectory $(\gamma_{v,s}^{(n)}(z))$ may affect the components $x^{[0]}+\dots + x^{[n-1]}$, as the following example shows. 

\begin{example}\label{ex-2}
We consider the degenerate Kolmogorov operator
\begin{eqnarray}\label{eq-K-ndi}
\mathcal{K}:= \partial_{x x}^2 +x \partial_y +x \partial_ x  -\partial_t.
\end{eqnarray}
In this case, $B$ takes the form
\begin{equation*}
 			   B = \begin{pmatrix}
 			       1 &  0   \\
 			       1 & 0 
   			 \end{pmatrix}.
\end{equation*}
and therefore the operator $\mathcal{K}$ is non dilation-invariant.
Let us emphasize the differences with the dilation-invariant case studied in Example \ref{simple_ex}. We denote again the points in $\mathbb{R}^3$ by $z=(x,y,t)$ and consider $\zeta = (0,0,0)$. The first two steps of the procedure used in Example \ref{simple_ex} allow us to move from $z$ to some point $z_1 = (x_1, y_1,0)$, then to some other point $z_2 = (0, y_2,0)$. The difference with the homogeneous case arises in the third step, i.e. when we are dealing with the $y$-variable.

Let us suppose we want to move from any point $z=(0,y,0) \in V_1$ to the origin $(0,0,0)$. If we reproduce the same construction as in \eqref{ex-traj}, we find:
\es{ex-traj2}{
z_1&=e^{s X}(z)=(s,y,0),\\
z_2&=e^{s^2 Y}(z_1)=(se^{s^2},-s+se^{s^2}+y ,-s^2),\\
z_3&=e^{-s X}(z_2)=(se^{s^2}-s,-s+se^{s^2}+y ,-s^2),\\
z_4&=e^{-s^2 Y}(z_3)=(s(1-e^{-s^2}),s(1-e^{-s^2}) +y,0).
}
If we choose $s$ such that $ s(1 -e^{-s^2})=-y$, we obtain $z_4= (-y,0,0)$, so that its second component is zero but, in constrast with the previous Example \ref{simple_ex}, we have that $z_4$ doesn't agree with our target point $\zeta = (0,0,0)$. 
\end{example}

In order to reach the point $\zeta = (0,0,0)$ also in the case of non dilation-invariant operators, we rely on the method introduced by Pagliarani and Pignotti in \cite{PagliaPigno}. In the case of the operator $\mathcal{K}$ in \eqref{eq-K-ndi} it is sufficient to use once more the integral curve of the vector field $X = \partial_x$. In the case of more general operators a further topological argument is needed to conclude the construction. We refer to \cite{PagliaPigno} for a detailed description of this construction.

We are now ready to prove our result.

\begin{proof}[Proof of Theorem \ref{taylor}]
 Let $z = (x,t), \z=(\xi, \tau)$ be two given points of $\Omega$. As explained above, the proof relies on a finite sequence of integral paths of the vector fields $X_1,\ldots,X_m$ and $Y$ connecting $z$ to $\z$. We use the construction made by Pagliarani, Pascucci and Pignotti \cite{PagliaPascPigno} for a dilation-invariant operator $\L$. In this case the trajectories $(\gamma_{v,s}^{(n)}(z))_{n=0,\ldots,\kappa}$ defined in \eqref{traj} are explicitely given and we prove that \eqref{tay-estimate} holds. We then discuss the modifications needed to deal with any non dilation-invariant operator $\L$, as introduced by Pagliarani and Pignotti in \cite{PagliaPigno}. 

As a preliminary result, we prove our claim \eqref{tay-estimate} under the  assumption that the points $z=(x,t)$ and $\zeta = (\xi, \tau)$ have the same temporal component $t=\tau$, by a finite iteration on $n = 0, \dots, \kappa$. We remove this assumption in the last part of the proof.

\medskip \noindent
\textbf{Base case $n=0$}. In this case, we are only changing the variables $x_{i}$, for $i=1,\ldots,m$, moving along the direction  $e^{s_0 X_{v_0}}$ where $v_0=(v_{0,1},\ldots,v_{0,m},0\ldots,0)$ is a suitable unit vector in $V_0$. Thus, equation \eqref{def-tay}, with $z=(x,t)$ and $\zeta = (x+s_0 v_0,t)$, rewrites as
\begin{eqnarray}\label{opic}
T^2_{z}u(\z) =u(x,t)+\sum_{i=1}^m \partial_{{x}_{i}} u(x,t) s_0 v_{0,i} +\frac{s_0^2}{2} \sum_{i,j=1}^m \partial ^2_{x_{i},x_{j}} u(x,t) v_{0,i} v_{0,j}.
\end{eqnarray}
We observe that $\Vert z^{-1} \circ \zeta \Vert ^2_K =|s_0|^2$ and therefore we want to show that 
\begin{eqnarray}\label{cask0}
u(\z)-T^2_{z}u(\z)=o(|s_0|^2) \quad \text{as} \quad s_0 \to 0.
\end{eqnarray}

By the multidimensional euclidean mean-value theorem, there exist $(\bar{v}_{i,j})_{1 \leq i,j \leq m}$, with $\bar{v}_{i,j} \in {\rm{span}} \lbrace e_{1},\ldots,e_{m} \rbrace$ and $|\bar{v}_{i,j}| \leq |v_0|$ such that
\begin{equation}\label{cask02}
\begin{split}
u(\z) - T^2_{z}u(\zeta)&=\frac{s_0^2}{2} \sum_{i,j=1}^m (\partial ^2_{x_{i},x_{j}} u(x+s_0 \bar{v}_{i,j},t) v_{0,i} v_{0,j}- \partial ^2_{x_{i},x_{j}} u(x,t)) v_{0,i} v_{0,j}\\
&=o(\vert s_0 \vert^2) \quad \textit{as $s_0 \to 0$},
\end{split}
\end{equation}
where we have used the continuity of the second order derivatives of $u$. Thus, we have proved \eqref{cask0} and we are done.

Let us remark that we don't need the dilation-invariance property for $Y$, as we don't make use of the vector field $Y$ in this part of the construction.

\medskip

\medskip \noindent
\textbf{Inductive step}. 
We now suppose the thesis true for a given non negative $n < \kappa$ and we prove it for $n+1$. 
For every $z, \zeta \in \R^{N+1}$ we set 
\begin{equation} \label{Ttilde}
 \tilde{T}^2_z u(\z):=T^2_z u(\z) - u(z).
\end{equation}
We define the points 
\en{
z&=(x,t),z_1=\gamma_{v,s}^{(n)}(z),z_2=e^{s^2 Y}(z_1)\\
z_3&=\gamma_{v,-s}^{(n)}(z_2),z_4=e^{-s^2 Y}(z_3)=\gamma_{v,s}^{(n+1)}(z)
}
where $v$ is the unique unitary vector in $V_{0,n+1} \subset V_0$, defined as $v=\frac{w}{|w|}$, where $w$ is the vector in $V_{0,n+1}$ such that $B^{n+1}w=\z^{[n+1]} -z^{[n+1]}$ and $s=|w|^{\frac{1}{2(n+1)+1}}$. 
We aim to prove that
\begin{eqnarray}\label{eq-3}
u(z_4)-T^2_{z}u(z_4)=o(\Vert z^{-1} \circ z_4 \Vert_K^2)=o(\vert s \vert^2) \quad \text{as} \quad s \to 0.
\end{eqnarray}
We now rewrite \eqref{eq-3} by using the notation \eqref{Ttilde} as follows
\es{boxes}{
u(z_4)-T^2_z u(z_4) &=\boxed{u(z_4)-u(z_3)}_{(1)}\\
&+\boxed{u(z_3)-u(z_2)-\tilde{T}^2_{z_2} u(z_3)}_{(2)} \\
&+ \boxed{u(z_2)-u(z_1)}_{(3)}\\
&+\boxed{\tilde{T}^2_{z_1} u(z) +u(z_1)-u(z)}_{(4)}\\
&+\boxed{\tilde{T}^2_{z_2} u(z_3) -\tilde{T}^2_{z_1} u(z)}_{(5)} \\
&-\boxed{\tilde{T}^2_{z} u(z_4)}_{(6)}.
}

By the inductive hypothesis, the second and the forth difference are $o(\vert s \vert^2)$ as $s \to 0$. 
Moreover, recalling \eqref{def-tay}, we have that $\tilde{T}^2_{z} u(z_4) \equiv 0$, being $x_4^{[0]} = x^{[0]}$ and $t_4 = t$. 

We next apply definition \eqref{def-tay} to the fifth difference, and we find
\begin{align}\label{eq-4}
\begin{split}
\tilde{T}^2_{z_2} u(z_3) -\tilde{T}^2_{z_1} u(z) =&-s \sum_{i=1}^{m} (\partial_{x_{i}} u(z_2) - \partial_{x_{i}} u(z_1))v_i \\
&-\frac{s^2}{2}\sum_{i,j=1}^{m}(\partial^2_{x_i x_j} u(z_2) - \partial^2_{x_i x_j} u(z_1))v_i v_j.
\end{split}
\end{align}
As a consequence of condition \eqref{eq-Ymix}, we obtain the following equation
\begin{equation*}
\partial_{x_{i}} u(z_2) - \partial_{x_{i}} u(z_1)=\partial_{x_{i}} u(e^{s^2 Y}(z_1))-\partial_{x_{i}} u(z_1)=o(|s|).
\end{equation*}
Using the previous equation and the continuity of second order derivatives of $u$, we find that \eqref{eq-4} is equal to $o(\vert s \vert^2)$.

We now observe that 
\begin{eqnarray*}
u(z_4)-u(z_3)=u(e^{-s^{2}Y}(z_3))-u(z_3).
\end{eqnarray*}
By applying the mean value theorem along the direction of the drift, we find that there exists $\bar{s}$ such that 
\begin{eqnarray*}
u(e^{-s^{2}Y}(z_3))-u(z_3)=-s^2 Yu(e^{\bar{s}Y}(z_3)),
\end{eqnarray*}
where $|\bar{s}| \leq |s|$.
Similarly we obtain that
\begin{eqnarray*}
u(z_2)-u(z_1)=s^2 Yu(e^{\tilde{s}Y}(z_1)),
\end{eqnarray*}
where again $\tilde{s}$ verifies $|\tilde{s}| \leq |s|$.

By letting $s \to 0$, we find that $\bar{s},\tilde{s} \to 0$, and therefore, using the continuity of $Yu$, we have showed that the sum of the first and the third difference in \eqref{boxes} is again equal to $o(\vert s \vert^2)$ as $s \to 0$.
This proves \eqref{eq-3} and therefore concludes the proof of the inductive step.

As already pointed out, the construction of the trajectories in the case of non dilation-invariant operators requires the adjustments introduced in \cite{PagliaPigno}, to deal with the fact that the term $\tilde{T}^2_{z} u(z_4)$ in \eqref{boxes} fails to vanish. Indeed, with the notation \eqref{eq-sommadiretta}, $x$ writes as $x = x^{[0]}+ x^{[1]}+ \dots+ x^{[\kappa]}$, and we have $\tilde{T}^2_{z} u(z_4) \ne 0$ whenever $x_4^{[0]}\neq x^{[0]}$. To overcome this problem, we define a new point $z_5 = \left(x_5, t_5\right)$ as follows:
\begin{eqnarray*}
 x_5^{[0]} = x^{[0]}, \quad x_5^{[1]} = x_4^{[1]}, \dots, x_5^{[\kappa]} = x_4^{[\kappa]}, \quad t_5 = t_4.
\end{eqnarray*}
Note that in \cite{PagliaPigno} it is proved that
\begin{eqnarray*}
 \left| x^{[0]} - x_4^{[0]} \right| \le C \Vert z^{-1} \circ \z \Vert_K,
\end{eqnarray*}
for some positive constant $C$ only depending on the matrix $B$. Then
 \begin{equation} \label{eq-z5z4}
u(z_5)-u(z_4) = o(\Vert z^{-1} \circ \z \Vert_K^2) \quad \text{as} \quad \zeta \to z.
\end{equation} 
With this modification, expression \eqref{boxes} is replaced by 
 \begin{equation*}
u(z_5)-T^2_z u(z_5) =u(z_5)-u(z_4) - \tilde{T}^2_{z} u(z_5) + o(\Vert z^{-1} \circ \z \Vert_K^2) \quad \text{as} \quad \zeta \to z.
\end{equation*} 
Moreover, $x_5^{[0]} =  x^{[0]}$ and $t_5 = t$ yield $\tilde{T}^2_{z} u(z_5) = 0$. From \eqref{eq-z5z4} it then follows that
 \begin{equation*}
u(z_5)-T^2_z u(z_5) = o(\Vert z^{-1} \circ \z \Vert_K^2) \quad \text{as} \quad \zeta \to z.
\end{equation*}

\medskip

We are now in position to prove \eqref{tay-estimate}. We first consider the point $\bar{z}=e^{(t-\tau)Y}(z)=(e^{(t-\tau)B} x,\tau)$ and write
\begin{eqnarray}\label{eq-5}
u(\z)-T^2_z u(\z)=u(\z)-T^2_{\bar{z}} u(\z)+T^2_{\bar{z}} u(\z)-T^2_z u(\z).
\end{eqnarray}
Thanks to the previous steps, the first difference is $o(\Vert \bar{z}^{-1} \circ \z \Vert^2_K)=o(\Vert z^{-1} \circ \z \Vert_K^2)$ as $\Vert z^{-1} \circ \z \Vert_K^2 \to 0 $, since $\zeta$ and $\bar{z}$ have the same temporal component $\tau$. At the same time, the second difference in \eqref{eq-5} can be rewritten as 

\begin{equation}\label{eq-6}
\begin{split}
T^2_{\bar{z}} u(\z)-T^2_z u(\z)&=u(\bar{z})-u(z)+\sum_{i=1}^m\left(\partial_{x_i}u(\bar{z})-\partial_{x_i}u(z)\right)(\xi_i-x_i)\\
&\quad +\frac{1}{2}\sum_{i,j=1}^m\left(\partial^2_{x_i x_j}u(\bar{z})-\partial^2_{x_i x_j}u(z)\right)(\xi_i-x_i)(\xi_j-x_j)+Yu(z)(\tau-t).
\end{split}
\end{equation}
Using the mean value theorem along the drift, we can rewrite difference $ u(\bar{z})-u(z)$ in \eqref{eq-6} as 
\begin{eqnarray}\label{diff1}
u(e^{(t-\tau)Y}(z))-u(z)=(t-\tau)Y u(e^{\delta Y}(z)),
\end{eqnarray}
where $\delta$ is such that $|\delta|\leq |t-\tau|$. Hence,  we obtain
\begin{eqnarray}
u(\bar{z})-u(z)-Y(z)(t-\tau)=(t-\tau)(Y u(e^{\delta Y}(z))-Y(z)),
\end{eqnarray}
which is $o(|t-\tau|)=o(\Vert z^{-1} \circ \z \Vert_K^2)$ as $\Vert z^{-1} \circ \z \Vert_K^2 \to 0 $, thanks to the continuity of $Yu$.

Finally, using again condition \eqref{eq-Ymix} and the continuity of the second derivatives of $u$, we obtain that the second and third difference in \eqref{eq-6} are also $o(|t-\tau|)=o(\Vert z^{-1} \circ \z \Vert_K^2)$ as $\Vert z^{-1} \circ \z \Vert_K^2 \to 0 $.

Combining all the previous estimates, we obtain
\begin{equation}
T^2_{\bar{z}} u(\z)-T^2_z u(\z)=o(\Vert z^{-1} \circ \z \Vert_K^2), \quad \textrm{as $\Vert z^{-1} \circ \z \Vert_K^2 \to 0 $}.
\end{equation}
and therefore \eqref{eq-5} is equal to $o(\Vert z^{-1} \circ \z \Vert_K^2) $ as $\Vert z^{-1} \circ \z \Vert_K^2 \to 0 $. This concludes the proof.
\end{proof}

\setcounter{equation}{0}\setcounter{theorem}{0}
\section{Proof of Theorem \ref{th-1} }
We first prove a preliminary lemma, which is a straightforward consequence of the maximum principle.
\begin{lemma}\label{lemma-max}
Given $\varphi \in C(\partial \H_R(z_0))$ and $g \in C_b(\H_R(z_0)) $, we let $v$ be the solution to the following Dirichlet problem
\begin{equation*}
\left\{ \begin{array}{ll}
\L v=g,\quad &\textit{in $\H_R(z_0)$},\\
v=\varphi,\quad &\textit{in $\partial \H_R(z_0)$}.
\end{array} \right.
\end{equation*}
Then, the following holds
\begin{eqnarray}\label{max-principle}
\Vert v \Vert_{{L}^\infty(\H_R(z_0))}\leq \Vert \varphi \Vert_{{L}^\infty(\H_R(z_0))} + |t-t_1|\Vert g \Vert_{{L}^\infty(\H_R(z_0))},
\end{eqnarray}
where $t_1 = t_0 - R^2$ is the time coordinate of the basis of the cylinder $\H_R(z_0)$.
\end{lemma}
\pr{
We introduce the function $w(x,t): =(t-t_1)\Vert g \Vert_{{L}^\infty(\H_R(z_0))}+\Vert \varphi \Vert_{{L}^\infty(\H_R(z_0))}$ and we let $u:=v-w$. Clearly, $u$ satisfies $\L u = g+\Vert g \Vert_{{L}^\infty(\H_R(z_0))} \geq 0$ in $\H_R(z_0)$. Moreover, as $v \equiv \varphi$ on the boundary of $ \H_R(z_0)$, we have $u=\varphi-(t-t_1)\Vert g \Vert_{{L}^\infty(\H_R(z_0))}-\Vert \varphi \Vert_{{L}^\infty(\H_R(z_0))}\leq 0$ in $\partial \H_R(z_0)$. By the strong maximum principle, it follows that $u(x,t) \leq 0$ in $\H_R(z_0)$. 
Replacing $v$ by $-v$, estimate \eqref{max-principle} follows at once.
}

\pr{[Proof of Theorem \ref{th-1}] We first prove assertion \emph{(ii)}. We denote $\H_{k}=\H_{\varrho^{k}}(0)$, $\varrho=\frac{1}{2}$ and we consider the following sequence of Dirichlet problems:
\begin{equation}\label{dirichlet1}
\left\{ \begin{array}{ll}
\L u_{k}=f(0),\quad \textit{in $\H_{k}$}\\
u_{k}=u,\quad \textit{in $\partial \H_{k}$}
\end{array} \right.
\end{equation}
For any point $z=(x,t)$ satisfying $\Vert z \Vert_K \leq \frac{1}{2}$, we want to estimate the quantity
\begin{eqnarray*}
I(z) := |\partial^2 u(z) - \partial^2 u(0)|,
\end{eqnarray*}
where $\partial^2 u(z)$ stands for either $\partial^2_{x_i x_j} u(z)$, with $i,j=1,\ldots,m$, or $Yu(z)$.
To this end, we write $I$ as the sum of three terms:
\begin{align*}
I(z) &\leq |\partial^2 u_{k}(z) - \partial^2 u_{k}(0)| + |\partial^2 u_{k}(0) - \partial^2 u(0)| + \\
&\qquad\qquad\quad+|\partial^2 u(z) - \partial^2 u_{k}(z)| =: I_{1}(z)+I_{2}(z)+I_{3}(z).
\end{align*}
We first estimate $I_{2}$. Following \cite{Wang}, we prove that $\left(\partial^2 u_{k}(0)\right)_{k \in \N}$ is a Cauchy sequence and that its limit agrees with $\partial^2 u(0)$. The same assertion holds for $I_3$ of course.

First, we let $v_{k} := u - u_{k}$ and we observe that $v_{k}$ satisfies the Dirichlet boundary value problem
\begin{equation}
\left\{ \begin{array}{ll}
\L v_{k}=f-f(0),\quad \textit{in $\H_{k}$}\\
v_{k}=0,\quad \textit{in $\partial \H_{k}$}
\end{array} \right.
\end{equation}\label{dirichlet2}
From Lemma \ref{lemma-max} it follows that
\begin{align}\label{maxconse1}
\|v_{k}\|_{\infty}\le 4 \varrho^{2k}\|f-f(0)\|_{\infty}\le 4 \varrho^{2k} \omega_f(\varrho^k).
\end{align}
Moreover, since $\L (u_k-u_{k+1})=0$  in $\H_{k+1}$, we apply Proposition \ref{corollary} and Lemma \ref{lemma-max}, and we find
\begin{align}\nonumber
\Vert \partial_{x_i}(u_{k}-u_{k+1}) \Vert_{L^{\infty}(\H_{k+2})} &\leq C\varrho^{-k-2}\sup_{\H_{k+1}}|u_{k}-u_{k+1}|\nonumber \\ \nonumber
&\leq C \varrho^{-k}\Big(\sup_{\H_{k+1}}|v_{k}|+\sup_{\H_{k+1}}|v_{k+1}|\Big)\\
&\leq C \varrho^{-k}\varrho^{2k}\omega_{f}(\varrho^{k})=C\varrho^{k}\omega_{f}(\varrho^{k}), \label{terza}
\end{align}
for any $i= 1, \dots, m$. In the same way, we obtain
\begin{align}\nonumber
\Vert \partial_{x_i x_j}^2(u_{k}-u_{k+1}) \Vert_{L^{\infty}(\H_{k+2})} &\leq C\varrho^{-2k-4}\sup_{\H_{k+1}}|u_{k}-u_{k+1}| \\ 
&\leq C \varrho^{-2k}\varrho^{2k}\omega_{f}(\varrho^{k})=C\omega_{f}(\varrho^{k}) \label{quar}
\end{align}
for $i,j= 1, \dots, m$, and
\begin{align}\nonumber
\Vert Y(u_{k}-u_{k+1}) \Vert_{L^{\infty}(\H_{k+2})} &\leq C\varrho^{-2k-4}\sup_{\H_{k+1}}|u_{k}-u_{k+1}| \\ 
&\leq C \varrho^{-2k}\varrho^{2k}\omega_{f}(\varrho^{k})=C\omega_{f}(\varrho^{k}). \label{quarta}
\end{align}
Let $k\ge 1$ such that $\varrho^{k+4} \le \Vert z\Vert_K \le \varrho^{k+3}$ , then we have:
\begin{equation}\label{stima_I_2}
\sum_{l=k}^{\infty} |\partial^2 u_l (0) - \partial^2 u_{l+1} (0)| \le C\sum_{l=k}^{\infty} \omega_f(\varrho^l)\le C \int_0^{\Vert z\Vert_K} \frac{\omega_f (r)}{r} dr.
\end{equation}

We next identify the sum of the series $\sum_{l=k}^{\infty} \left(\partial^2 u_l (0) - \partial^2 u_{l+1} (0)\right)$ as 
\begin{equation}\label{sum-series}
 \sum_{l=k}^{\infty} \left(\partial^2 u_l (0) - \partial^2 u_{l+1} (0)\right) = \partial^2 u_k (0) - \partial^2 u (0).
\end{equation}
To this aim, we first consider the derivative $ \partial^2_{x_i x_j}u_k$ and we prove that
\begin{eqnarray} \label{conv-Taylor}
\lim_{k \to +\infty} \partial^2_{x_i x_j} u_k(0)=\partial^2_{x_i x_j} T_0^2u(0), 
\end{eqnarray}
where $T_0^2u(\z)$ is the second-order Taylor polynomial of $u$ around the origin, computed at some point $\z=(\x, \t) \in \H_k$:
\begin{eqnarray*}
T_0^2u (\z)= u(0) + \sum_{i=1}^m \partial_{x_{i}} u(0) \x_{i}
+ \frac{1}{2}\sum_{i,j=1}^m \partial ^2_{x_{i}x_{j}} u(0) \x_{i} \x_{j} - Yu(0)\t.
\end{eqnarray*}
Thus, by applying Theorem \ref{taylor} to $u\in C^{2}_\L(\H_{1}(0))$, we obtain from \eqref{conv-Taylor} that
\begin{eqnarray} \label{conv-u}
\lim_{k \to +\infty} \partial^2_{x_i x_j} u_k(0)=\partial^2_{x_i x_j} u(0).
\end{eqnarray}

We compute $\L T_0^2 u$ in $\zeta=(\xi,\tau)$ as 
\begin{equation*}
\begin{split}
\L T_0^2 u(\zeta)& = \! \! \sum_{i,j=1}^m \partial^2_{\xi_i \z_j} u(0) - \partial_t u(0) + \sum_{j=1}^N 
\! \! \bigg( \! \sum_{i=1}^m b_{ij} \partial_{\xi_i} u(0) \bigg) \xi_j + \sum_{l,j=1}^N 
\! \! \bigg( \! \sum_{i=1}^m b_{il} \partial^2_{\xi_j \xi_i }u(0) \bigg) \xi_l \xi_j\\
&= \! \! \sum_{i,j=1}^m \partial^2_{\xi_i \z_j} u(0) - \partial_t u(0) + \langle v, \xi \rangle + \langle  M \xi ,\xi \rangle,
\end{split}
\end{equation*}
where $v =\left(v_j \right)_{j=1,\ldots,N} = \left(\sum_{i=1}^m b_{ij} \partial_{\xi_i} u(0)\right)_{j=1,\ldots,N}$ is a constant vector of $\R^N$ and $M=\left(m_{lj}\right)_{l,j=1,\ldots,N} =\left(\sum_{i=1}^m b_{il} \partial^2_{\xi_j \xi_i }u(0)\right)_{l,j=1,\ldots,N}$ is a $N \times N$ constant matrix. 

In addition, as $\L u = f$ in $\H_k$, we have that
\begin{eqnarray}
\sum_{i,j=1}^m \partial^2_{\xi_i \z_j} u(0) - \partial_t u(0)=\L_0 u(0) =\L u(0) =f(0)
\end{eqnarray}
and thus
\begin{eqnarray}
\L T_0^2 u(\zeta)=f(0) + \langle v, \xi \rangle + \langle M \xi ,\xi \rangle.
\end{eqnarray}
Thus, the definition of $u_k$ in \eqref{dirichlet1} gives us
\begin{eqnarray}
\L \left( T_0^2 u-u_k\right)(\zeta)= \langle v, \xi \rangle + \langle M \xi ,\xi \rangle, \quad \zeta \in \H_k.
\end{eqnarray}
We now apply Lemma \ref{lem-tec-2} to $T_0^2 u-u_k$ for $R=\varrho^k$ and infer
\begin{equation}\label{tay-uk}
 |\partial^2_{x_i x_j}(u_k - T_0^2 u)(0)| \leq C \varrho^{-2k}\sup_{\H_k} \vert u_k - T_0^2 u \vert + O(\varrho^k).
\end{equation}

Moreover, since $T_0^2u$ is the second-order Taylor polynomial of $u$, we have $u(\z) = T_0^2u(\z) + o(\Vert \z \Vert_K^2)$. It follows that 
\begin{eqnarray}\label{esti-with-tay}
\sup_{\z \in \H_k}|u - T_0^2u| = o(\varrho^{2k})
\end{eqnarray}
Thus, from estimates \eqref{esti-with-tay} and \eqref{maxconse1}, we obtain
\begin{eqnarray}\label{sup-est}
\sup_{\H_k}|u_k - T_0^2u| \leq \sup_{\H_k}|v_k| +\sup_{\H_k}|u - T_0^2u| \leq 4\omega_{f}(\varrho^k)\varrho^{2k} + o(\varrho^{2k}) \leq o(\varrho^{2k}).
\end{eqnarray}
Estimates \eqref{tay-uk} and \eqref{sup-est} finally yield
\begin{equation*}
|\partial^2_{x_i x_j}(u_k - T_0^2u)(0)| \leq C \varrho^{-2k} \sup_{\H_k}|u_k - T_0^2u| +O(\varrho^k)\leq C \varrho^{-2k}o(\varrho^{2k}) +O(\varrho^k) \leq o(1),
\end{equation*}
where, as usual, the indexes $i$ and $j$ range from $1$ to $m$. Thus, for any $i,j=1,\ldots,m$ we have showed that \eqref{conv-Taylor} holds true. Repeating the same argument for the vector field $Y$, and using again Theorem \ref{taylor}, we obtain:
\begin{eqnarray*}
\lim_{k \to +\infty}Y u_k(0)=Y T_0^2u(0)=Y u(0).
\end{eqnarray*}
In conclusion, using \eqref{stima_I_2}, we obtain:
\begin{equation}\label{stima_I_2-b}
I_2 \le \sum_{l=k}^{\infty} |\partial^2 u_l (0) - \partial^2 u_{l+1} (0)| \le C \int_0^{\Vert z\Vert_K} \frac{\omega_f (r)}{r} dr,
\end{equation}
for $k\ge 1$ such that $\varrho^{k+4} \le \Vert z\Vert_K \le \varrho^{k+3}$.
Similarly, we can estimate $I_3$ through the solution of $\L v=f(z)$ in $\H_j(z)$ and $v=u$ on $\partial \H_j(z)$ and obtain
\begin{equation}\label{stima_I_3}
I_3 \le \sum_{l=k}^{\infty} |\partial^2 u_l (z) - \partial^2 u_{l+1} (z)| \le C \int_0^{\Vert z\Vert_K} \frac{\omega_f (r)}{r} dr.
\end{equation}

Finally, let us estimate $I_1$. 
Since $h_k=u_k-u_{k+1} \in C^\infty(\H_{k+2})$,  we can apply Proposition \ref{mean-value-lem} to the functions $\partial^2_{x_i x_j} h_k$ and $Y h_k$:
\en{
|\partial^2_{x_i x_j} h_k(z)-\partial^2_{x_i x_j} h_k(0)| \leq \frac{C}{\varrho^{k}}\Vert z \Vert_K \Vert \partial^2_{x_i x_j} \Vert_{L^{\infty}(\H_{k+1})}
}
and 
\en{
|Y h_k(z)-Y h_k(0)| \leq \frac{C}{\varrho^{k}}\Vert z \Vert_K \Vert Y h_k \Vert_{L^{\infty}(\H_{k+1})},
}
for $i,j=1,\ldots,m$. We can now apply once again \eqref{quar} to obtain
\en{
|\partial^2_{x_i x_j} h_k(z)-\partial^2_{x_i x_j} h_k(0)| \leq \frac{C}{\varrho^{k}}\Vert z \Vert_K \Vert \partial^2_{x_i x_j} h_k \Vert_{L^{\infty}(\H_{k+1})}
\leq C\Vert z \Vert_K\varrho^{-k}\omega_{f}(\varrho^{k}).
}
In addition, thanks to \eqref{quarta}, we infer
\en{
|Y h_k(z)-Y h_k(0)| \leq \frac{C}{\varrho^{k}}\Vert z \Vert_K \Vert Y h_k \Vert_{L^{\infty}(\H_{k+1})}
\leq C\Vert z \Vert_K\varrho^{-k}\omega_{f}(\varrho^{k}) .
}
Hence, since $u_k(z)-u_k(0)=u_0(z)-u_0(0)+\sum_{j=0}^{k-1}\left(h_j(0)-h_j(z)\right)$, we have
\begin{align*}
I_1 &\le | \partial^2 u_0 (z)-\partial^2u_0 (0)|+\sum_{j=0}^{k-1} |\partial^2 h_j (z)-\partial^2 h_j(0)|\\
&\le C \Vert z \Vert_K \big( \|u_0\|_{L^\infty(\H_0)}  +C\sum_{j=0}^{k-1} \varrho^{-j} \omega_f(\varrho^j)\big)\\
&\le C\Vert z \Vert_K \big( \|u\|_{L^\infty(\H_1(0))} + \|f\|_{L^\infty(\H_1(0))}+ C \int_{\Vert z \Vert_K} ^1 \frac{\omega_f(r)}{r^2}\big).
\end{align*}
Combining the above estimate with \eqref{stima_I_2-b} and \eqref{stima_I_3}, we complete the proof of \emph{(ii)}.

\medskip

We now prove assertion \emph{(i)}. We consider $u_1$ solution to the following Dirichlet problem
\begin{eqnarray*}
\left\{ \begin{array}{ll}
\L u_{1}=f(0),\quad &\textit{in $\H_{1/2}(0)$}\\
u_{1}=u,\quad &\textit{in $\partial \H_{1/2}(0)$}
\end{array} \right.
\end{eqnarray*}
Then, we have
\begin{eqnarray}\label{par-u-0}
|\partial^2 u(0)| \leq |\partial^2 u(0) -\partial^2 u_1(0)| + |\partial^2 u_1(0)|
\end{eqnarray}
Thanks to \eqref{stima_I_2-b}, we can estimate the first term in \eqref{par-u-0} as
\begin{eqnarray}
|\partial^2 u(0) -\partial^2 u_1(0)|\le C \int_0^{1} \frac{\omega_f (r)}{r} dr.
\end{eqnarray}
To estimate the second term in \eqref{par-u-0}, we consider the function $v(z):=u_1(z)\eta_{1/2}(z)$, where $\eta_{1/2} $ is the cut-off function introduced in \eqref{cut-off} with $R=\frac{1}{2}$. Reasoning as in the proof of Proposition \ref{lem-apriori}, we obtain
\begin{equation*}
\begin{split}
u(z) = v(z) &= \int_{\H_{\frac{1}{2}}(0)}[\Gamma(z,\cdot) \div(A D_ m(\eta_{1/2})) u_1](\z)d \z\\ 
&\quad- \int_{\H_{\frac{1}{2}}(0)}[\Gamma(z,\cdot) Y(\eta_{1/2})u_1](\z)d\z \\
&\quad - \int_{\H_{\frac{1}{2}}(0)}[\Gamma(z,\cdot)\eta_{1/2}\L(u_1)](\z)d\z\\
&+ 2 \int_{\H_{\frac{1}{2}}(0)}[\langle D_m^\zeta \Gamma(z,\cdot), A D_m \eta_{1/2} \rangle u_1](\z)d\z,
\end{split}
\end{equation*}
where $z \in \H_{\frac{1}{4}}(0)$. Thanks to Lemma \ref{lemma-max}, we estimate
\begin{eqnarray*}
\sup_{\H_{\frac{1}{2}}(0)}{|u_1|} \leq \sup_{\H_{\frac{1}{2}}(0)}{|u|} +4|f(0)|. 
\end{eqnarray*}
As the derivatives of $\eta_{1/2}$ vanish in $\H_{3/8}(0)$, for any $i,j=1,\ldots,m$, we obtain
\begin{equation}\label{overlineI}
\begin{split}
|\partial^2_{x_i x_{j}}u_1(z)|
&\leq \int_{\H_{\frac{1}{2}}(0) \setminus \H_{\frac{3}{8}}(0)}\big\vert[\partial^2_{x_i x_{j}} \Gamma(z,\cdot) \div(A D_ m(\eta_{1/2})) u_1](\z)\big\vert d\z \\ 
&\quad + \int_{\H_{\frac{1}{2}}(0) \setminus \H_{\frac{3}{8}}(0)}\big\vert[\partial^2_{x_i x_{j}} \Gamma(z,\cdot) Y(\eta_{1/2}) u_1](\z)\big\vert d\z \\  
&\quad +2 \int_{\H_{\frac{1}{2}}(0) \setminus \H_{\frac{3}{8}}(0)}\big\vert[\langle \partial^2_{x_i x_{j}} D_m^\zeta \Gamma(z,\cdot), A D_m \eta_{1/2}\rangle u_1] (\z)\big\vert d\z \\
&\quad +\bigg\vert f(0)\Big[\partial^2_{x_i x_{j}}\int_{\H_{\frac{1}{2}}(0)} [ \Gamma(z,\cdot) \eta_{1/2} ](\z) d\z \Big] \bigg\vert\\
&\quad =: \overline{I}_{1}(z) + \overline{I}_{2}(z)+\overline{I}_3(z)+\overline{I}_4(z).
\end{split}
\end{equation}
Moreover, as the derivatives of $\eta_{1/2}$ are bounded, we estimate the first and second integral in \eqref{overlineI} as
\begin{eqnarray*}
\overline{I}_{1}(z) \leq C \big[ \sup_{\H_{\frac{1}{2}}(0)}{|u|} +4|f(0)|\big],\\ 
\overline{I}_{2}(z) \leq C \big[ \sup_{\H_{\frac{1}{2}}(0)}{|u|} +4|f(0)|\big],\\
\overline{I}_{3}(z) \leq C \big[ \sup_{\H_{\frac{1}{2}}(0)}{|u|} +4|f(0)|\big].
\end{eqnarray*}
Finally, by taking advantage of \eqref{est-sec-gamma-const}, we obtain that $\overline{I}_4(z)$ is bounded by a constant $C$ that only depends on $B$, $\lambda$ and $\Lambda$.

By using the same argument we can estimate $|Y u_1(0)|$ and thus
\begin{eqnarray}\label{par-u-0-2}
|\partial^2 u_1(0)| \leq C \big[ \sup_{\H_{\frac{1}{2}}(0)}{|u|} +4|f(0)|\big].
\end{eqnarray}
Combining estimates \eqref{par-u-0} and \eqref{par-u-0-2}, we conclude the proof of Theorem \ref{th-1}.
}

\setcounter{equation}{0}\setcounter{theorem}{0}
\section{ Dini continuous coefficients}
This Section is devoted to the proof of Theorem \ref{th-2}. We therefore consider a solution $u$ to the equation
\begin{equation*} 
\LL u=f,
\end{equation*}
where the operator $\LL$ does satisfy the hypotheses {\rm \bf[H.1]} and {\rm \bf[H.2]} and $f$ is assumed to be Dini continuous, and we proceed as in the proof of Theorem \ref{th-1}. Specifically, we denote $\H_{k}=\H_{\varrho^{k}}(0)$, $\varrho=\frac{1}{2}$ and we consider the following sequence of Dirichlet problems:
\begin{equation} \label{dirichlet3}
\left\{ \begin{array}{ll}
\sum\limits_{i,j=1}^m a_{ij}(0,0) \partial^2_{x_i x_j} u_{k}+Yu_{k}=f(0), \quad \text{in $\H_{k}$}\\
u_{k}=u,\quad \text{on $\partial \H_{k}$}.
\end{array} \right.
\end{equation}
Note that the bounds given in Propositions \ref{lem-apriori}, \ref{corollary} and \ref{mean-value-lem} only depend on the constants $\lambda, \Lambda$ in {\rm \bf[H.2]} and on the matrix $B$. Keeping in mind this fact, the proof of Theorem \ref{th-2} is given by the same argument used in the proof of Theorem \ref{th-1}.

\begin{proof}[Proof of Theorem \ref{th-2}]
Consider, for every $k \in \N$, the auxiliary function $v_{k} := u - u_{k}$, and note that it is a solution to the boundary value problem
\begin{equation} \label{dirichlet4}
\left\{ \begin{array}{lll}
\sum\limits_{i,j=1}^m a_{ij}(0,0) \partial^2_{x_i x_j} v_k +Y v_k \\
\quad =f-f(0)+\sum\limits_{i,j=1}^m (a_{ij}(0)-a_{ij}(x,t))\partial^2_{x_i x_j}  u, \quad \text{in $\H_{k}$}\\
v_k =0,\quad \textit{in $\partial \H_{k}$}
\end{array} \right.
\end{equation}
In order to simplify the notation, we let
\begin{equation} \label{eta-def}
    \eta := \max_{i,j=1,\dots,m}\|\partial^2_{x_i x_j} u\|_{L^\infty(\H_{1})}.
\end{equation}
From Lemma \ref{lemma-max} it follows that
\begin{align*}
\|v_{k}\|_{L^\infty(\H_k)} \le C\varrho^{2k} [\omega_f(\varrho^k)+\omega_a(\varrho^k) \eta].
\end{align*}
Hence
\begin{align*}
\|u_{k}-u_{k+1}\|_{L^\infty(\H_{k+1})} \le C\varrho^{2k} [\omega_f(\varrho^k)+\omega_a(\varrho^k) \eta].
\end{align*}

As already observed, we can apply Corollary \ref{corollary} and obtain estimates for the second order derivatives of $v_k$. In fact, for any $i,j=1,\ldots,m$, we have

\begin{align}\nonumber
\Vert \partial^2_{x_i x_j}(u_{k}-u_{k+1}) \Vert_{L^{\infty}(\H_{k+2})} &\leq C(\varrho^{k})^{-2}\sup_{\H_{k+1}}|u_{k}-u_{k+1}| \\ 
&\leq C \varrho^{-2k}\varrho^{2k}[\omega_{f}(\varrho^{k})+\omega_a(\varrho^k) \eta]=C[\omega_{f}(\varrho^{k}) +\omega_a(\varrho^k) \eta]
\end{align}
and
\begin{align}\nonumber
\Vert Y(u_{k}-u_{k+1}) \Vert_{L^{\infty}(\H_{k+2})} &\leq C(\varrho^{k})^{-2}\sup_{\H_{k+1}}|u_{k}-u_{k+1}| \\ 
&\leq C \varrho^{-2k}\varrho^{2k}[\omega_{f}(\varrho^{k})+\omega_a(\varrho^k) \eta]=C[\omega_{f}(\varrho^{k}) +\omega_a(\varrho^k) \eta]
\end{align}

To estimate the second order derivatives of the function $u$, we apply Theorem \ref{taylor} and proceed as in the proof of Theorem \ref{th-1}. Since there are no significant differences, we omit the details here.
\end{proof}

\end{document}